\documentclass[11pt]{article}
\usepackage{amsmath, amssymb, amsfonts}
\usepackage{amsthm}

\usepackage[mathlines,pagewise]{lineno}
\usepackage{enumitem}
\usepackage{cite}
\newtheorem{lemma}{Lemma}           
\usepackage{helvet}
\usepackage{graphicx}
\usepackage{hyperref}
\usepackage{geometry}
\usepackage{graphicx}
\usepackage{caption}
\usepackage{chngcntr}
\usepackage{titlesec}
\usepackage{float} 
\usepackage{todonotes}

\DeclareCaptionLabelSeparator{pipe}{\ \textbar\ }
\DeclareCaptionLabelSeparator{pipe}{\ \textbar\ }
\renewcommand\thesubsection{\arabic{section}.\arabic{subsection}\hspace{-0.7em}}

\title{Solving and Learning Advective Multiscale Darcian Dynamics with the Neural Basis Method}

\author{
Yuhe Wang\thanks{%
\hangindent=1.8em\hangafter=1\relax
Asia Pacific Technology, Inc., Houston, TX 77042, USA;
Institute for Scientific Computation, Texas A\&M University, College Station, TX 77843, USA.
\ \texttt{yuhe.wang@me.com}%
}
\and
Min Wang%
\thanks{Department of Mathematics, University of Houston, Houston, TX 77204, USA. \texttt{mwang55@central.uh.edu}}%
}

\date{}

\begin{document}
\renewcommand{\thefootnote}{\fnsymbol{footnote}}
\setcounter{footnote}{0}

\renewcommand{\thefootnote}{\fnsymbol{footnote}}
\setcounter{footnote}{0}

\maketitle
\vspace{-1em}


\begin{abstract}
\noindent
Physics--governed models are increasingly paired with machine learning for accelerated computations, yet most ``physics--informed'' formulations treat the governing equations as a penalty loss whose scale and meaning are set by heuristic balancing. 
This blurs operator structure, thereby confounding solution approximation error with governing-equation enforcement error and making the solving and learning progress hard to interpret and control. 
Here we introduce the Neural Basis Method, a projection-based formulation that couples a predefined, physics-conforming neural basis space with an operator-induced residual metric to obtain a well-conditioned deterministic minimization. 
Stability and reliability then hinge on this metric: the residual is not merely an optimization objective but a computable certificate tied to approximation and enforcement, remaining stable under basis enrichment and yielding reduced coordinates that are learnable across parametric instances.
We use advective multiscale Darcian dynamics as an engineering demonstration to illustrate this broader point. 
Our method produces accurate and robust solutions in single solves and enables fast and effective parametric inference for essential engineering tasks such as control, optimization and uncertainty quantification. 
\vspace{1em}  

\end{abstract}

\section{\textbf{Introduction}}
\vspace{-1.ex} 
Flow--transport coupling in multiscale porous materials \cite{moyner2019multiscale,farber2025porous} poses major challenges for modeling, especially in parametric many-query workflows \cite{alcalde2018estimating,krevor2023subsurface}.
Governed by elliptic–parabolic and advective PDEs with heterogeneous coefficients, such multiphysics dynamics entail conservation constraints, nonlinear effects, and multiscale interactions \cite{hunt2017flow,vanden2007heterogeneous}. 
Physics-informed machine learning \cite{karniadakis2021physics}, such as physics-informed neural networks (PINNs) \cite{raissi2019physics,wang2023expert}, is often presented as promising because it offers mesh-free approximations, seamless data integration, and potential learning-enabled acceleration. 
In practice, however, most approaches enforce physics through dimensionless $L^2$ loss minimization with heuristic weighting, which obscures physical scaling and disrupts
numerical structure. This can yield poorly conditioned optimization in which approximation error is entangled with operator inconsistency, undermining accuracy and reliability \cite{wang2021understanding,bischof2025multi,krishnapriyan2021characterizing,wu2023comprehensive, wang2022and}. 
This raises a fundamental question: where should physics and numerical structure reside in learning-based PDE models? 

In this work, we propose the Neural Basis Method (NBM), which departs from loss-driven training by seeking solutions through an explicit, deterministic procedure. 
We use predefined physics-conforming neural networks to construct finite-dimensional approximation spaces.
Then, we impose PDEs by formulating a well-defined projection problem and solve for the solution coefficients. 
This design aligns NBM with classical numerical analysis and allows us to incorporate variational formulations \cite{xu2025weak} in a systematic and principled manner. 
Furthermore, for parametric cases, we represent solution families across parameter instances within the same neural space by learning the parameter-dependent solution coefficients, which naturally yields an operator learning formulation. 
This perspective is especially useful in many-query engineering settings under varying coefficients, boundary conditions, or control inputs.

NBM is motivated by two lines of ideas: projection-based methods (e.g., Galerkin), which use fixed finite-dimensional space to enable stability, error control, and systematic refinement 
\cite{brenner2008mathematical,babuvska1978posteriori}, and neural approximation approaches, which offer expressive spaces for representing solution manifolds \cite{cybenko1989approximation,barron2002universal,yarotsky2017error}.
NBM therefore aims to integrate neural representations while retaining the structure and guarantees of projection-based schemes.  
Related ideas have recently appeared, including random feature and extreme learning machine-type methods, where fixed single-hidden-layer neural networks serve as basis functions and solution coefficients are obtained by least-squares \cite{huang2006extreme,chen2022bridging,rahimi2007random,dong2021local,zhang2024transferable,shang2023randomized,shang2024randomized,sun2024local}. 
They provide an important proof-of-concept, but the non-orthogonal bases often yield severely ill-conditioned least-squares systems \cite{chen2011minimizing,zhang2025fourier,van2026local,chen2024high}. 
Especially, for multiscale problems, resolving fine-scale structures requires basis enrichment \cite{chen2022bridging,hu2025morphology,zhang2025shallow,chi2024random}, which can drive rapid, often exponential, condition-number growth. 
NBM builds on this direction while departing in three ways: a multilayer neural basis generator that resists conditioning degradation under basis enrichment; an essential physical structure that is built in at vector field representation level; and a weighting and stabilization strategy that enforces consistency with the underlying operator and physical scaling.
Our design draws on insights from least-squares finite element methods (LSFEM) and the discontinuous Petro-Galerkin (DPG) framework. 
LSFEM inspires properly scaled boundary residuals in seeking solutions via operator-residual minimization \cite{bochev2009least,starke1999multilevel,cai1994first,cai1997first}.
DPG theory connects optimal test norms for multiscale and advective problems to generalized least-squares formulations \cite{demkowicz2010class,demkowicz2011class}.  
Collectively, these insights motivate the NBM view that weighting and stabilization are treated as principled components of the projection-based approach rather than as heuristic loss balancing, preserving operator consistency and physical scaling and enabling reliable residual-based measures. 

A direct consequence of this projection-based perspective is that NBM provides a built-in, physically meaningful measure of solution quality. 
This stands in contrast to the penalty losses used in many physics-informed neural approaches, whose magnitude and minimization depend sensitively on heuristic balancing and offer limited insight into operator consistency. 
In NBM, the residual serves as a numerically interpretable certificate tied to the discrete PDE enforcement, supporting systematic reasoning about approximation error.
Thus, residual reduction can be interpreted within a projection framework analogous to classical schemes, enabling stability analysis and error assessment in the spirit of Galerkin theory, including Céa-type approximation properties. 
This residual-based notion of measure underpins the reliability of both solution accuracy and subsequent parametric inference built on the learned representation.

We then extend NBM from single instance solves to parametric operator learning (NBM-OL).  
In many-query engineering tasks, including control, optimization, and uncertainty quantification, the goal is to learn how the solution operator varies with governing parameters \cite{lu2021learning,wang2021learning,kovachki2023neural}. 
This necessitates a measure that is numerically comparable across parameters and, under stable projection, provides a computable a posteriori bound on the solution error \cite{grepl2005posteriori, qiu2025variationally}.  
In NBM-OL, we learn parametric dependence in the fixed neural basis space equipped with such measure, instead of surrogate losses that lack direct numerical meaning. 
This provides a robust utility for monitoring training and steering refinement across hyperparameters, with well-informed stopping criterion. 

In this study, we present the first complete formulation and demonstration of NBM/NBM-OL in the setting of advective multiscale Darcian dynamics, with carbon storage as a representative application.  
We use physics-conforming neural vector bases in a nonlinear mixed Darcy flow formulation with energy-consistent weighting, and advance transport using a purely advective update stabilized by upwind control-volume treatment.
As an introductory paper, we adopt a collocation-based realization for clarity and efficiency, although the method also supports variational formulations.
Numerical experiments show accurate and stable single-instance solutions, and, for parametric inference, excellent in-distribution generalization, strong out-of-distribution robustness, and massive acceleration.
In representative examples, we observe computational speedups on the order of $10^3$--$10^4$, and the gain scales with the prediction temporal horizon and problem size. 

This paper is organized as follows. In the setting of coupled Darcy flow--transport, we first present the core ingredients of NBM and evaluate its performance on single-instance problems. We then turn to parametric many-query regimes, where variability arises from permeability fields and boundary conditions, and show how NBM-OL enables efficient parametric inference. Together, we illustrate the central principle of NBM: by placing physics in the neural discretization through physics-conforming approximation spaces and operator-induced residual metrics, the method unifies direct PDE solving and parametric operator learning within the same projection-based formulation.

\section{\textbf{Overview of NBM as a neural PDE solver}}
\vspace{-1.5ex}
The overall workflow of NBM for coupled multiscale Darcy flow--transport (see Appendix~\ref{Appx:A}) is summarized in Fig.~\ref{fig:fig1}. 
It brings together the main components of the collocation-based implementation, including the underlying physical system, neural basis approximations for scalar (pressure, concentration) and vector fields (mass-flux), the mixed weighted least-squares formulation for the multiscale Darcy flow solution with energy-consistent weighting of interior and boundary residuals, and the implicit time-marching procedure for the coupled flow--transport dynamics. 
\begin{figure}[ht!]
    \centering
    \includegraphics[width=1\linewidth]{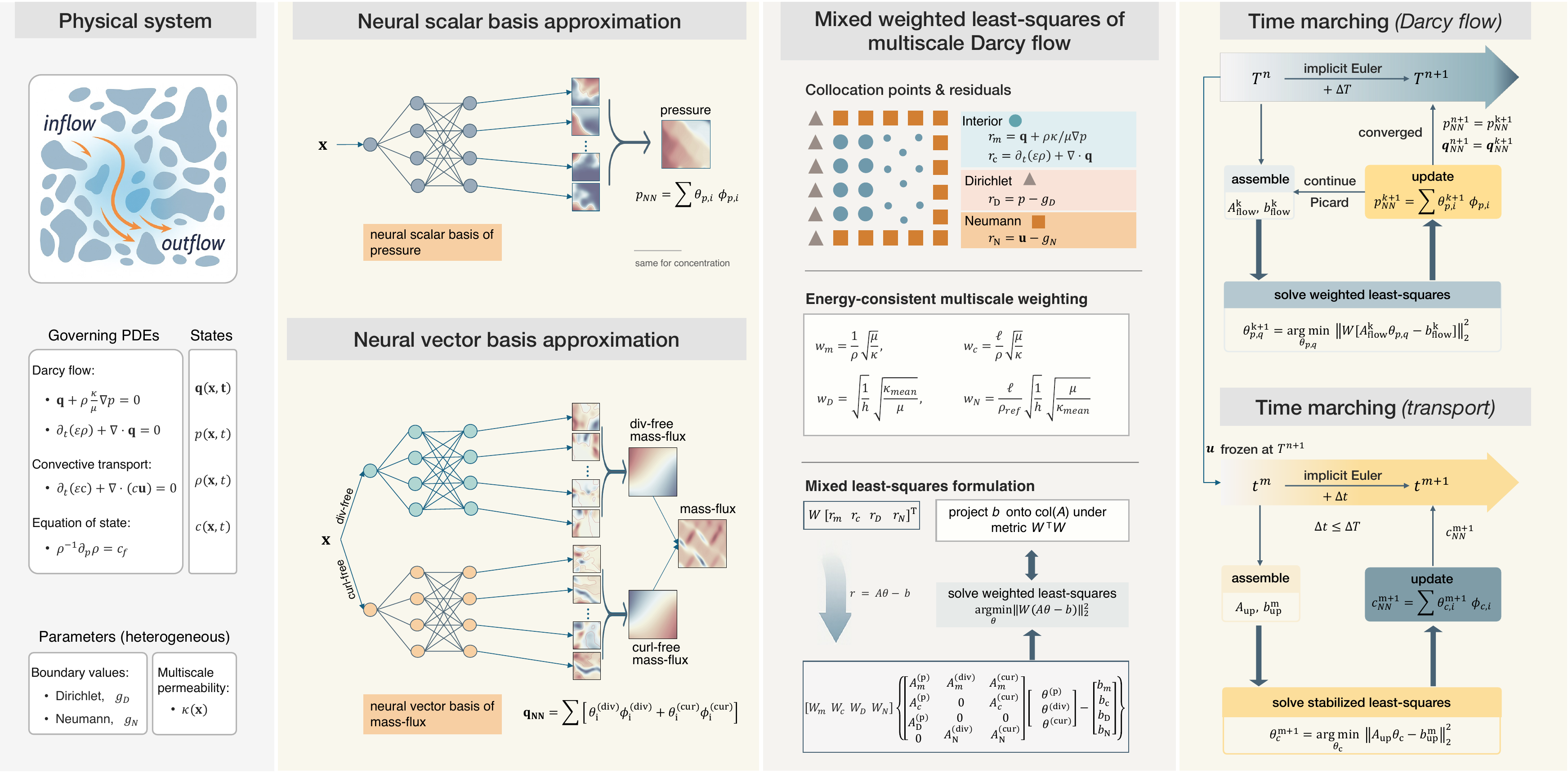}
    \caption{%
    \bfseries Overview of NBM as a neural PDE solver for coupled multiscale Darcy flow--transport.\,
    \bfseries Left, \normalfont Physical system.\,
    \bfseries Middle-left, \normalfont Neural basis approximation for scalar and vector field.\,
    \bfseries Middle-right, \normalfont Mixed weighted least-squares projection for multiscale Darcy flow.\,
    \bfseries Right, \normalfont Time marching of Darcy flow and transport.%
    }
    \label{fig:fig1}
\end{figure}
\vspace{-1.5ex}
\subsection{Core NBM formulation}
\vspace{-1.5ex}
NBM approximates solutions with finite neural basis functions and computes the solution coefficients through a PDE-constrained least-squares projection. 
For system \eqref{eq:momentum}-\eqref{eq:mass}, the pressure $p(\mathbf{x},t)$ and mass-flux $\mathbf{q}(\mathbf{x},t)$ are smooth and therefore suit global basis expansion naturally. 
In contrast, the concentration $c(\mathbf{x},t)$ can develop limited regularity with discontinuous fronts and thus requires specialized stabilization or regularization. 
In general, we write $s_{NN}(\mathbf{x},t)=\sum \theta_i(t)\,\phi_i(\mathbf{x})$, where a multilayer, fixed-parameter, ResNet-style network \cite{he2016deep} generates the scalar spatial bases $\{\phi_i(\mathbf{x})\}$. Moreover, we let the coefficient vector $\boldsymbol{\theta}(t)$ carry temporal evolution.
Here $s$ denotes a scalar field, namely pressure, concentration, or a mass-flux component such as $q_x$ or $q_y$.
We freeze network weights/biases, so we defines a deterministic approximation space $\mathrm{span}\{\phi_i\}$ and obtain the solution by projecting onto this space rather than by training the network.
Specifically, we use two hidden layers to balance expressiveness and numerical conditioning while maintaining a clear geometric interpretation of the induced basis functions. We refer to this construction as a dual-layer neural basis.

For the vector field $\mathbf{q}$, we propose a neural basis design inspired by the Helmholtz decomposition \cite{robinson2016three,wang2009robust} such that we approximate via a sum of a divergence-free part $\boldsymbol{\phi}_i^{(\mathrm{div})}$ and curl-free part $\boldsymbol{\phi}_i^{(\mathrm{div})}$,
\begin{equation}
\mathbf{q}_{NN}
=
\sum
\Big[
\theta_i^{(\mathrm{div})}\,\boldsymbol{\phi}_i^{(\mathrm{div})}
+
\theta_i^{(\mathrm{curl})}\,\boldsymbol{\phi}_i^{(\mathrm{curl})}
\Big].
\label{eq:vector_basis}
\end{equation}
We build the curl-free part as $\boldsymbol{\phi}_i^{(\mathrm{curl})}=\nabla\phi_i$, which represents the irrotational flow driven by potential.
For the divergence-free part, we use a streamfunction formulation in 2D,
$\boldsymbol{\phi}_i^{(\mathrm{div})}=\nabla^{\perp}\phi_i$, and a vector potential formulation in 3D,
$\boldsymbol{\phi}_i^{(\mathrm{div})}=\nabla\times\mathbf{F}_i$, 
where $\mathbf{F}_i$ can be constructed from the same scalar basis $\{\phi_i\}$.
By this design, $\nabla\cdot\boldsymbol{\phi}_i^{(\mathrm{div})}=0$ holds identically, so only the curl-free subspace enters equation \eqref{eq:mass}; it supplies the degrees of freedom needed to represent compressible effects and source-driven mass redistribution, while the divergence-free component naturally captures mass-preserving motion.
This physics-conforming representation preserves the intrinsic structure of fluid motion as much as possible at the representation level and can effectively improve solution accuracy. See Appendices~\ref{Appx:B}-\ref{Appx:E} for details of construction and properties of neural basis functions.

To compute the coefficient vector $\boldsymbol{\theta}$ at each time step or linearization iterate, NBM enforces \eqref{eq:pde_system} via minimizer
\(
\boldsymbol{\theta}^*
=
\arg\min_{\boldsymbol{\theta}}
\left\|
\mathcal{L}\!\left[\sum \theta_i \phi_i\right]
\right\|_{\Omega}^{2}
\;+\;
\left\|
\mathcal{B}\!\left[\sum \theta_i \phi_i\right]
- g
\right\|_{\partial\Omega}^{2},
\)
which is linear in $\{\theta_i\}$ with $\{\phi_i\}$ being fixed. After neural discretization, it becomes
$
\boldsymbol{\theta}^*
=
\arg\min_{\boldsymbol{\theta}}
\;\big\| \mathbf{A}\,\boldsymbol{\theta} - \mathbf{b} \big\|_\mathbf{V}^2 ,
$
where $\mathbf{A}$ and $\mathbf{b}$ encode the discretized PDE operators and boundary conditions under certain time stepping or linearization, and $\mathbf{V}$ denote a properly-chosen norm in a general sense.
In collocation-based realization we enforce the PDEs at collocation points.
This minimization admits a projection interpretation: its residual is orthogonal to $\mathrm{col}(\mathbf{A})$ under
$\langle \cdot,\cdot\rangle_{\mathbf{V}}$, so we enforce the PDE constraints via a least-squares projection in the neural basis space. 
In practice, NBM removes nonconvex training altogether: at each update, $\boldsymbol{\theta}$ is obtained by solving a deterministic linear or linearized least-squares system, retaining the operator structure of classical schemes.
Furthermore, under standard continuity and stability assumptions, this projection admits a probabilistic Céa-type bound: the error is controlled by the best approximation in the $N_b$-dimensional neural basis space, plus a high-probability $N_b^{-1/2}$ term, which is a conservative baseline guarantee provided by the following theoretical analysis.

\begin{lemma}[Probabilistic Céa-type bound for neural basis spaces]
\label{lem:probabilistic_cea}
Let $V$ be a Hilbert space equipped with norm $\|\cdot\|_V$, and let
$a(\cdot,\cdot):V\times V\to\mathbb{R}$ be a continuous bilinear form with
continuity constant $C_a>0$ and stability constant $\gamma>0$ (either coercive or inf--sup stable).
Let $\{\phi_{\omega}\}_{\omega\sim\pi}\subset V$ denote a family of fixed neural basis functions
parameterized by random variables $\omega$ drawn i.i.d.\ from a distribution $\pi$, and define the
$N_b$-dimensional neural basis space
\[
V_{N_b}:=\mathrm{span}\{\phi_{\omega_1},\ldots,\phi_{\omega_{N_b}}\}\subset V,
\qquad \omega_j \stackrel{\mathrm{i.i.d.}}{\sim}\pi,\ \ j=1,\ldots,N_b .
\]
Let $u\in V$ be the exact solution of
\[
a(u,v)=\ell(v),\qquad \forall v\in V,
\]
and let $u_{N_b}\in V_{N_b}$ denote the least-squares projection solution associated with the NBM
formulation. Assume that $u$ belongs to the reproducing kernel Hilbert space $\mathcal{H}_k$
induced by the random neural basis construction, and that $\mathcal{H}_k$ is continuously embedded in $V$,
\[
\|w\|_V \le C_{\mathrm{emb}}\|w\|_{\mathcal{H}_k},\qquad \forall w\in \mathcal{H}_k,
\]
where $C_{\mathrm{emb}}$ is the embedding constant of $\mathcal{H}_k\hookrightarrow V$.
Then, for any $\delta\in(0,1)$, there exists a constant $C_1(\delta)>0$ such that, with probability at least $1-\delta$,
\[
\|u-u_{N_b}\|_V
\le \frac{C_a}{\gamma}\left(
\inf_{w\in \mathcal{H}_k}\|u-w\|_V
+ C_{\mathrm{emb}}\,C_1(\delta)\,N_b^{-1/2}
\right).
\]
If, in addition, $u\in\mathcal{H}_k$ with $\|u\|_{\mathcal{H}_k}\le B$, then
\[
\|u-u_{N_b}\|_V
\le \frac{C_a}{\gamma}\,C_{\mathrm{emb}}\,C_1(\delta)\,B\,N_b^{-1/2}.
\]
\end{lemma}

\begin{proof}
By the Céa estimate associated with the Galerkin or
least-squares projection, the NBM solution $u_{N_b}\in V_{N_b}$ satisfies
\begin{equation}
\|u-u_{N_b}\|_V
\;\le\;
\frac{C_a}{\gamma}\inf_{w\in V_{N_b}}\|u-w\|_V .
\label{eq:cea}
\end{equation}
Since $u\in\mathcal H_k$, classical results on Monte--Carlo approximation of
kernel expansions (e.g., Rahimi--Recht or Bach-type bounds) imply that, for any
$\delta\in(0,1)$, with probability at least $1-\delta$, there exists a function
$v_{N_b}\in V_{N_b}$ such that
\begin{equation}
\|u-v_{N_b}\|_{\mathcal H_k}
\;\le\;
C_1(\delta)\,\|u\|_{\mathcal H_k}\,N_b^{-1/2}.
\label{eq:mc}
\end{equation}
Here $v_{N_b}$ is a comparison function in $V_{N_b}$ whose existence is guaranteed by
random feature approximation theory and $C_1(\delta)$ is a constant depends on $\delta$ but is independent of $N_b$.
Using the continuous embedding $\mathcal H_k\hookrightarrow V$, we obtain
\begin{equation}
\|u-v_{N_b}\|_V
\;\le\;
C_{\mathrm{emb}}\,\|u-v_{N_b}\|_{\mathcal H_k}.
\label{eq:embedding}
\end{equation}
Combining \eqref{eq:mc} and \eqref{eq:embedding} yields
\begin{equation}
\|u-v_{N_b}\|_V
\;\le\;
C_{\mathrm{emb}}\,C_1(\delta)\,\|u\|_{\mathcal H_k}\,N_b^{-1/2}.
\label{eq:vnV}
\end{equation}
Since $v_{N_b}\in V_{N_b}$, by definition of the infimum,
\begin{equation}
\inf_{w\in V_{N_b}}\|u-w\|_V
\;\le\;
\|u-v_{N_b}\|_V .
\label{eq:inf}
\end{equation}
Substituting \eqref{eq:inf} into the C\'ea's estimate \eqref{eq:cea} and using
\eqref{eq:vnV} gives
\[
\|u-u_{N_b}\|_V
\;\le\;
\frac{C_a}{\gamma}\,
C_{\mathrm{emb}}\,C_1(\delta)\,\|u\|_{\mathcal H_k}\,N_b^{-1/2}.
\]
More generally, without assuming a bound on $\|u\|_{\mathcal H_k}$, the same
argument yields
\[
\|u-u_{N_b}\|_V
\;\le\;
\frac{C_a}{\gamma}
\left(
\inf_{w\in \mathcal H_k}\|u-w\|_V
+
C_{\mathrm{emb}}\,C_1(\delta)\,N_b^{-1/2}
\right).
\]
\end{proof}
\vspace{-2.ex} 
The above bound does not rely on symmetry of the operator.
For the transport equation, the bilinear form
$a(\cdot,\cdot)$ is generally non-symmetric but satisfies an inf--sup
condition under appropriate least-squares norms.
In such cases, the constant $\gamma$ represents the corresponding
inf--sup stability constant.
Since the NBM formulation employs a deterministic least-squares realization, which can be interpreted as a Petrov--Galerkin
projection, the same Céa-type estimate applies.
\vspace{-1.5ex}
\subsection{Mixed least-squares with energy-consistent weights for multiscale Darcy flow} 
\vspace{-1.5ex}
Discrete neural gradient and divergence operators are generally not adjoint. Consequently, a pressure-only approximation that eliminates $\mathbf{q}$ through \eqref{eq:momentum} and then recovers it a posteriori does not, in general, preserve local mass conservation when inserted into \eqref{eq:mass}. We therefore devise a mixed formulation and solve $\boldsymbol{\theta}_p$ and $\boldsymbol{\theta}_q$ simultaneously using least-squares for each update by
\begin{equation}
\begin{aligned}
\boldsymbol{\theta}_{p,q}^*
=
\arg\min_{\boldsymbol{\theta}_{p,q}}
\Bigg\{
&
\left\|
\mathbf{q}_{NN}
+
\rho\,\frac{\kappa}{\mu}\,\nabla p_{NN}
\right\|_{\Omega}^{2}
+
\left\|
\frac{\partial(\varepsilon\rho(p_{NN}))}{\partial t}
+
\nabla\cdot\mathbf{q}_{NN}
\right\|_{\Omega}^{2}
\\[4pt]
&
+
\left\|
p_{NN} - p_D
\right\|_{\partial\Omega_p}^{2}
+
\left\|
\mathbf{q}_{NN}\cdot\mathbf{n} - q_N
\right\|_{\partial\Omega_q}^{2}
\Bigg\}.
\end{aligned}
\label{eq:nbm_mixed_darcy}
\end{equation}
Here $\boldsymbol{\theta}_{p,q}$ denotes $(\boldsymbol{\theta}_p, \boldsymbol{\theta}_q)$, $p_{NN}=\sum_i\theta_{p,i}\,\phi_{p,i}$ and $\mathbf{q}_{NN}=\sum_i\theta_{q,i}\,\phi_{q,i}$, a compact form of \eqref{eq:vector_basis}.
This mixed projection jointly enforces the constitutive law, the continuity constraint, and the boundary conditions, and produces locally conservative fluxes without a posterior correction. 
Additionally, $\|\cdot\|_{\Omega/\partial\Omega}$ denotes a generic norm defined in $\Omega$ or on $\partial\Omega$, and its construction hinges on how we measure the interior and boundary residuals. 

Unweighted $L^2$ norms do not respect the Darcian energy scaling, which can degrade accuracy in multiscale media; this motivates energy-consistent weighting.
Accordingly, we place all residual terms on a common Darcian energy scale: we weight the constitutive residual to match the dissipative metric
$\int_{\Omega}\frac{\mu}{\kappa}\,|\mathbf{u}|^2\,d\Omega$,
and scale the continuity and boundary residuals to the same metric.
Under collocation realization, this yields:
\begin{equation}
\begin{aligned}
\boldsymbol{\theta}_{p,q}^*
=
\arg\min_{\boldsymbol{\theta}_{p,q}}
\Bigg\{
&
\left\|
w_m\,
\Big(
\mathbf{q}_{NN}
+
\rho\,\frac{\kappa}{\mu}\,\nabla p_{NN}
\Big)
\right\|_{L^2(\Omega)}^{2}
+
\,\, \left\|
w_c\,
\Big(
\frac{\partial(\varepsilon\rho(p_{NN}))}{\partial t}
+
\nabla\cdot\mathbf{q}_{NN}
\Big)
\right\|_{L^2(\Omega)}^{2}
\\[4pt]
&
+
\left\|
w_D\,
\big(
p_{NN} - p_D
\big)
\right\|_{L^2(\partial\Omega_p)}^{2}
\,\, +
\,\, \left\|
w_N\,
\big(
\mathbf{q}_{NN}\cdot\mathbf{n} - q_N
\big)
\right\|_{L^2(\partial\Omega_q)}^{2}
\Bigg\},
\end{aligned}
\label{eq:nbm_mixed_darcy_weighting}
\end{equation}
with
\begin{equation}
w_m = \frac{1}{\rho}\sqrt{\frac{\mu}{\kappa}},
\qquad
w_c = \frac{\ell}{\rho}\sqrt{\frac{\mu}{\kappa_{\mathrm{mean}}}},
\qquad
w_D = \sqrt{\frac{1}{h}}\sqrt{\frac{\kappa_{\mathrm{mean}}}{\mu}},
\qquad
w_N = \frac{\ell}{\rho_{\mathrm{0}}}\sqrt{\frac{1}{h}}\sqrt{\frac{\mu}{\kappa_{\mathrm{mean}}}}.\label{eq:ls-weight}
\end{equation}
In $w_c$ and $w_N$, $\ell$ is a characteristic length that ensures dimensional consistency relative to the constitutive residual. 
The weights $w_D$ and $w_N$ balance the pressure and mass-flux boundary residuals against the interior energy scaling with a measure correction $\sqrt{1/h}$ to compensate for the codimension-one boundary sampling in collocation. Here $h$ is the collocation spacing of the residual evaluation over the physical domain.
$\kappa_{\mathrm{mean}}$ is the mean permeability and $\rho_{0}$ is the reference density. These weights measure all terms in compatible weighted $L^2$ norms and yield a consistent, physically interpretable projection. 

In implementation, each residual evaluation at a collocation point corresponds to one row of the discrete weighted least-squares system. 
For interior points $x_i\in\Omega$ with associated volume $|\Omega_i|$, we use
\(
w_m(x_j)=1/\rho \, \sqrt{\mu/\kappa(x_j)}\,\sqrt{|\Omega_j|}
\)
for the constitutive residual and
\(
w_c(x_j)=\ell/\rho\,\sqrt{\mu/\kappa_\mathrm{mean}}\,\sqrt{|\Omega_j|}
\)
for the continuity residual.
For boundary points $x_b\in\partial\Omega$ with associated length or area $|\partial\Omega_b|$, we set
\(
w_D(x_b)=\sqrt{1/h}\,\sqrt{\kappa_{\mathrm{mean}}/\mu}\,\sqrt{|\partial\Omega_b|}
\)
and
\(
w_N(x_b)=\ell/\rho_{ref}\,\sqrt{1/h}\,\sqrt{\mu/\kappa_{\mathrm{ref}}}\,\sqrt{|\partial\Omega_b|}.
\) 
Accordingly, the discrete residual vector takes the stacked form
\[
\mathbf{r}(\boldsymbol{\theta})
=
\begin{bmatrix}
\begin{array}{l}
\mathbf{r}_m\\
\mathbf{r}_c\\
\mathbf{r}_D\\
\mathbf{r}_N
\end{array}
\end{bmatrix}
=
\begin{bmatrix}
\begin{array}{l}
\mathbf{A}_m\\
\mathbf{A}_c\\
\mathbf{A}_D\\
\mathbf{A}_N
\end{array}
\end{bmatrix}
\boldsymbol{\theta}
-
\begin{bmatrix}
\begin{array}{l}
\mathbf{b}_m\\
\mathbf{b}_c\\
\mathbf{b}_D\\
\mathbf{b}_N
\end{array}
\end{bmatrix},
\qquad
\boldsymbol{\theta}
=
\begin{bmatrix}
\begin{array}{l}
\boldsymbol{\theta}^{(p)}\\
\boldsymbol{\theta}^{(\mathrm{div})}\\
\boldsymbol{\theta}^{(\mathrm{cur})}
\end{array}
\end{bmatrix},
\]
where $\mathbf{r}_m\in\mathbb{R}^{M_\Omega}$ and $\mathbf{r}_c\in\mathbb{R}^{M_\Omega}$ collect the constitutive and continuity residuals at interior collocation points, and $\mathbf{r}_D\in\mathbb{R}^{M_D}$ and $\mathbf{r}_N\in\mathbb{R}^{M_N}$ collect the Dirichlet and Neumann residuals at boundary collocation points. The corresponding discrete system has a block structure
\[
\underbrace{\begin{bmatrix}
\mathbf{A}_m^{(p)} & \mathbf{A}_m^{(\mathrm{div})} & \mathbf{A}_m^{(\mathrm{cur})} \\
\mathbf{A}_c^{(p)} & \mathbf{0}                   & \mathbf{A}_c^{(\mathrm{cur})} \\
\mathbf{A}_D^{(p)} & \mathbf{0}                   & \mathbf{0} \\
\mathbf{0}         & \mathbf{A}_N^{(\mathrm{div})} & \mathbf{A}_N^{(\mathrm{cur})}
\end{bmatrix}}_{\mathbf{A}_{\mathrm{flow}}\in\mathbb{R}^{M_{\mathrm{tot}}\times N_\theta}}
\begin{bmatrix}
\begin{array}{l}
\boldsymbol{\theta}^{(p)}\\
\boldsymbol{\theta}^{(\mathrm{div})}\\
\boldsymbol{\theta}^{(\mathrm{cur})}
\end{array}
\end{bmatrix}
-
\underbrace{\begin{bmatrix}
\begin{array}{l}
\mathbf{b}_m\\
\mathbf{b}_c\\
\mathbf{b}_D\\
\mathbf{b}_N
\end{array}
\end{bmatrix}}_{\mathbf{b}_{\mathrm{flow}}\in\mathbb{R}^{M_{\mathrm{tot}}}},
\qquad
M_{\mathrm{tot}}=2M_\Omega+M_D+M_N,
\]
where each row of $\mathbf{A}_m,\mathbf{A}_c,\mathbf{A}_D,\mathbf{A}_N$ is obtained by evaluating the bases and their derivatives at the corresponding collocation point and inserting them into the constitutive, continuity, and boundary residual definitions. Finally, we have the following diagonal weighting matrix
\[
\mathbf{W}
=
\mathrm{diag}\!\Big(
\mathbf{w}_m,\mathbf{w}_c,\mathbf{w}_D,\mathbf{w}_N
\Big)
\in\mathbb{R}^{M_{\mathrm{tot}}\times M_{\mathrm{tot}}},
\]
where $\mathbf{w}_m$ and $\mathbf{w}_c$ collect the interior weights $w_m(x_j)$ and $w_c(x_j)$, while $\mathbf{w}_D$ and $\mathbf{w}_N$ collect the Dirichlet and Neumann boundary weights $w_D(x_b)$ and $w_N(x_b)$. We therefore write
\begin{equation}
\boldsymbol{\theta}_{p,q}^*
= \arg\min_{\boldsymbol{\theta}_{p,q}}
\left\|
\mathbf{W}
\bigl(
\mathbf{A}_{\mathrm{flow}}\,\boldsymbol{\theta}_{p,q}
-
\mathbf{b}_{\mathrm{flow}}
\bigr)
\right\|_2^2. \label{eq:nbm_mixed_darcy_weighting_compact}
\end{equation}
We then define a residual measure in a relative sense as $\mathcal{E}_{\mathrm{rel}}=\big\|\mathbf{W}
\bigl( \mathbf{A}_\mathrm{flow}\,\boldsymbol{\theta}_{p,q} - \mathbf{b}_\mathrm{flow} \bigr)\big\|_2^2/{\big\|\mathbf{b}_\mathrm{flow} \big\|_2^2}$.
\vspace{-1.5ex}
\subsection{Upwind control-volume stabilization for transport}
\vspace{-1.5ex}
For transport, we treat the velocity field $\mathbf{u}_{NN}$ as known and hold it frozen over each Darcy-flow update. 
To stabilize advection, we enforce \eqref{eq:transport} using a control-volume collocation with first-order upwinding. 
This gives the following discrete balance for each control-volume $K$
\begin{equation}
\varepsilon\,
\frac{c_{NN,K}^{m+1}-c_{NN,K}^{m}}{\Delta t}\,|K|
\;+\;
\sum_{f\subset\partial K}
(\mathbf{u}_{NN}\cdot\mathbf{n})_f\,c_{NN,f}^{m+1, \mathrm{up}}\,|f|
=
0,
\label{eq:transport_fvm}
\end{equation}
where $c_{NN}=\sum_i \theta_{c,i}\,\phi_{c,i}$, $|K|$ denotes the measure of the cell $K$, $|f|$ denotes the measure of a face $f\subset\partial K$, and $m$ is the time index for transport update.
The sign of $(\mathbf{u}_{NN}\cdot\mathbf{n})_f$ selects the upwind state and defines $c_{NN,f}^{m+1,\mathrm{up}}$.
We assemble each face flux by evaluating the neural basis $\phi_{c,i}$ in the upwind cell.
The inflow boundary concentration then enters \eqref{eq:transport_fvm} naturally as the corresponding upwind state, 
eliminating explicit boundary residuals.
This construction yields a linear least-squares solve for each concentration update.
By inheriting the stability of classical upwinding finite-volume schemes while retaining a global low-dimensional neural representation, it effectively suppresses the spurious oscillations that global bases often introduce in advective regimes. 
Similarly, we write its compact form under collocation
\begin{equation}
\boldsymbol{\theta}_{c}^*
= \arg\min_{\boldsymbol{\theta}_{c}}
\left\|
\mathbf{A}_{\mathrm{up}}\,\boldsymbol{\theta}_{c}
-
\mathbf{b}_{\mathrm{up}}
\right\|_2^2, \label{eq:nbm_transport_compact}
\end{equation}
where $\mathbf{A}_{\mathrm{up}}$ and $\mathbf{b}_{\mathrm{up}}$ are assembled by applying the upwind control-volume balance \eqref{eq:transport_fvm} to each cell and its residual measure is thus $\mathcal{E}_{\mathrm{rel}}=\big\| \mathbf{A}_\mathrm{up}\,\boldsymbol{\theta}_c - \mathbf{b}_\mathrm{up} \big\|_2^2/{\big\|\mathbf{b}_\mathrm{up} \big\|_2^2}$. 
This construction defines an advection-aware residual metric and can be interpreted as an implicit, generalized weighting, analogous in spirit to the energy-consistent weighting used for Darcy flow.
\vspace{-1.5ex}
\subsection{Implicit time integration and nonlinearity treatment.}
\vspace{-1.5ex}
For Darcy flow, given $p_{NN}^n$ at time $T^n$, we advance one Darcy time step $\Delta T$ to $T^{n+1}$ by solving \eqref{eq:nbm_mixed_darcy_weighting_compact} with implicit time integration. For the nonlinearity induced by pressure-dependent density in \eqref{eq:momentum}, we choose Picard linearization. At Picard iterate $k$, we update the coefficients by
\begin{equation}
\boldsymbol{\theta}_{p,q}^{k+1}
=
\arg\min_{\boldsymbol{\theta}_{p,q}}
\left\|
\mathbf{W}
\bigl(
\mathbf{A}_{\mathrm{flow}}^{k}\,\boldsymbol{\theta}_{p,q}
-
\mathbf{b}_{\mathrm{flow}}^{k}
\bigr)
\right\|_2^2 ,
\label{eq:nbm_darcy_picard}
\end{equation}
where $\mathbf{A}_{\mathrm{flow}}^{k}$ and $\mathbf{b}_{\mathrm{flow}}^{k}$ are evaluated using the previous iterate $p_{NN}^{k}$. With
$p_{NN}^{k+1}=\sum \theta_{p,i}^{k+1}\phi_{p,i}$ and
$\mathbf{q}_{NN}^{k+1}=\sum \theta_{q,i}^{k+1}\boldsymbol{\phi}_{q,i}$,
we iterate until convergence and obtain
$(p_{NN}^{n+1},\mathbf{q}_{NN}^{n+1})$ at $T^{n+1}$.
At the end of each Darcy step, we then advance transport from $T^n$ to $T^{n+1}$ with frozen velocity
$\mathbf{u}^{n+1}=\mathbf{q}_{NN}^{n+1}/\rho_{NN}^{n+1}$.
We use implicit time integration to overcome CFL restriction. 
However, to better resolve the advective front propagation, we take multiple transport substeps within $[T^n,T^{n+1}]$ because concentration generally evolves on much smaller time scale than Darcy flow. 
Consequently, for each transport substep $t^{m}\!\to t^{m+1}$ in $[T^n,T^{n+1}]$, we formulate \eqref{eq:nbm_transport_compact} as
\begin{equation}
\boldsymbol{\theta}_c^{\,m+1}
=
\arg\min_{\boldsymbol{\theta}_c}
\;
\big\|
\mathbf{A}_{\mathrm{up}}\,\boldsymbol{\theta}_c
-
\mathbf{b}_{\mathrm{up}}^{m}
\big\|_2^2 ,
\label{eq:nbm_transport_upwind}
\end{equation}
where $\mathbf{A}_{\mathrm{up}}$ is fixed, while $\mathbf{b}_{\mathrm{up}}^{\,m}$ collects the previous-step result together with the boundary inflow condition. Then, when $t^{m+1}=T^{n+1}$, we set
$c^{n+1}=c^{m+1}=\sum \theta_{c,i}^{\,m+1}\phi_{c,i}$.

\section{\textbf{Performance of NBM on single-query problems}}
\vspace{-1.5ex}
In this section, we assess the performance of NBM on representative single-query Darcy--transport problems. We focus on solution accuracy, residual decay, basis refinement behavior, and robustness under increasing multiscale complexity. We begin with a homogeneous compressible benchmark to illustrate the basic approximation and convergence behavior of NBM, and then move to challenging multiscale permeability fields to examine the role of energy-consistent weighting and the robustness of the method in heterogeneous regimes.
\vspace{-1.5ex}
\subsection{NBM produces advective Darcian solutions with spectral-like accuracy}
\vspace{-1.5ex}
For a homogeneous compressible Darcy flow--transport problem based on supercritical CO$_2$ storage (Fig.~\ref{fig:fig2}a), the residual $\mathcal{E}_{\mathrm{rel}}$ decays in two near-linear stages on semi-log scale as the neural basis size $N_b$ increases: a rapid drop that resolves dominant structures, followed by a milder slope as higher-frequency features become harder to represent with a finite basis (Fig.~\ref{fig:fig2}b).
Consistent with this behavior, the relative $L_2$ errors of pressure, velocity, and concentration decrease rapidly and remain competitive with FVM across all variables (Fig.~\ref{fig:fig2}c).
The velocity statistics also converge in a spectral-like manner, as quantified by the Kolmogorov--Smirnov distance, and the qualitative fields show systematic refinement with increasing $N_b$ (Fig.~\ref{fig:fig2}d--e).
For fieldwise errors in pressure, velocity, and concentration, we compare NBM against a finite-volume method (FVM) and a vanilla PINN. NBM attains a $0.001\%$ relative $L_2$ error in pressure, about $50\%$ lower than FVM ($0.002\%$). For velocity, NBM yields $0.16\%$ for $u_x$ (FVM: $0.14\%$) and $0.30\%$ for $u_y$ (FVM: $0.33\%$), whereas the vanilla PINN underperforms ($2.65\%$ for $u_x$ and $4.43\%$ for $u_y$). For concentration, NBM ($5.17\%$) closely matches FVM ($5.36\%$), while the PINN degrades severely ($100\%$) (Fig.~\ref{fig:fig2}f–i). The experiment details of this example are provided in Appendix~\ref{Appx:F}.
\begin{figure}[ht!]
  \centering
  \includegraphics[width=\linewidth]{figs/fig2.jpg}
  \caption{%
    \bfseries NBM performance on a homogeneous compressible benchmark.\,
    \bfseries a,
    \normalfont Homogeneous CO$_2$ storage model. 
    \bfseries b,
    \normalfont Residual $\mathcal{E}_{\mathrm{rel}}$ vs. neural basis size $N_b$.
    \bfseries c,
    \normalfont Relative $L_2$ errors vs. $N_b$.
    \bfseries d,
    \normalfont Kolmogorov--Smirnov distance vs. $N_b$. 
    \bfseries e,
    \normalfont Refinement of pressure and concentration solutions with $N_b$.
    \bfseries f,
    \normalfont Left to right: ground-truth pressure, followed by absolute error fields (w.r.t.\ ground truth) for NBM, FVM, and vanilla PINN.
    \bfseries g,
    \normalfont Left to right: ground-truth velocity magnitude, followed by absolute error fields (w.r.t.\ ground truth) for NBM, FVM, and vanilla PINN.
    \bfseries h,
    \normalfont Left to right: ground-truth concentration, followed by absolute error fields (w.r.t.\ ground truth) for NBM, FVM, and vanilla PINN.
    \bfseries i,
    \normalfont Relative $L_2$ errors.
    \bfseries Note:\normalfont\ All results are at 90 days. In f--i, NBM uses $N_b=1000$, and all methods (NBM, FVM, and vanilla PINN) employ a $50\times 50$ spatial resolution. Pressure is reported in $\mathrm{bar}$, velocity in $\mathrm{m\,day^{-1}}$, and concentration in $\mathrm{ppm}$.
  }
  \label{fig:fig2}
\end{figure}
\vspace{-15pt}

\begin{figure}[!hb]
  \centering
  \includegraphics[width=\linewidth]{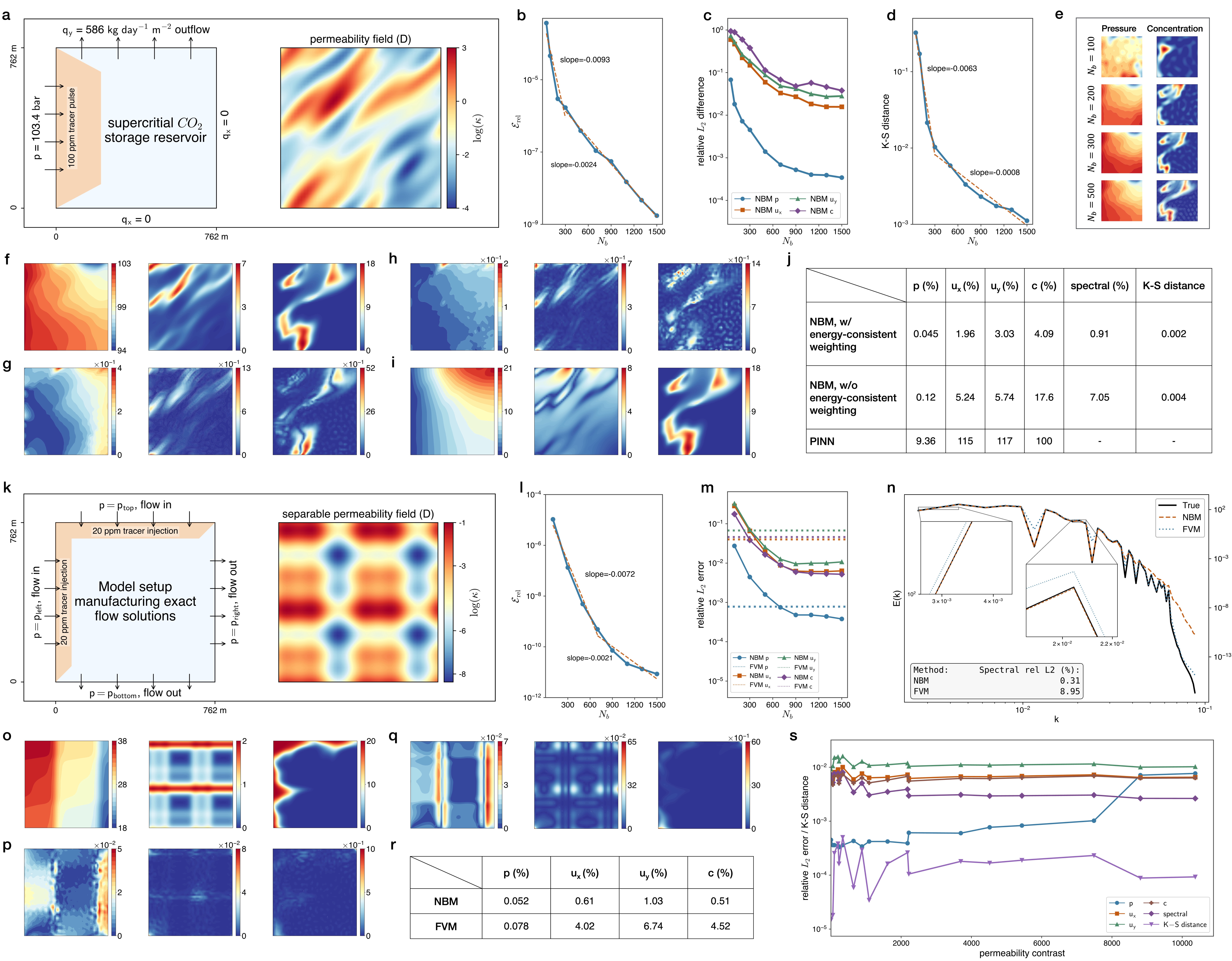}
    \caption{%
    \bfseries NBM performance on multiscale benchmarks.\,
    \bfseries a, \normalfont Multiscale CO$_2$ storage model.  (permeability contrast $502$).\, 
    \bfseries b, \normalfont  Residual $\mathcal{E}_{\mathrm{rel}}$ vs. neural basis size $N_b$.\,
    \bfseries c, \normalfont Relative $L_2$ difference vs. $N_b$.\,
    \bfseries d, \normalfont K-S distance vs. $N_b$.\,
    \bfseries e, \normalfont Refinement of NBM pressure and concentration fields with increasing $N_b$.\,
    \bfseries f, \normalfont Left to right: baseline pressure, velocity magnitude, and concentration.\, 
    \bfseries g, \normalfont Left to right: absolute differences of NBM with energy-consistent weighting, for pressure, velocity magnitude, and concentration.\,
    \bfseries h, \normalfont Left to right: absolute differences of NBM without energy-consistent weighting, for pressure, velocity magnitude, and concentration.\,
    \bfseries i, \normalfont Left to right: absolute differences of PINN, for pressure, velocity magnitude, and concentration.s.\,
    \bfseries j, \normalfont Summary metrics measured against the baseline.\,
    \bfseries k, \normalfont Manufactured-solution setup (permeability contrast 1078).\,
    \bfseries l, \normalfont Residual $\mathcal{E}_{\mathrm{rel}}$ vs. $N_b$.\,
    \bfseries m, \normalfont Relative $L_2$ error vs. $N_b$.\,
    \bfseries n, \normalfont Energy spectrum $E(k)$ comparison.\,
    \bfseries o, \normalfont Ground-truth pressure, velocity magnitude, and concentration fields.\,
    \bfseries p, \normalfont Left to right: absolute errors of NBM, for pressure, velocity magnitude, concentration.\,
    \bfseries q, \normalfont Left to right: absolute errors of FVM, for pressure, velocity magnitude, concentration.\,
    \bfseries r, \normalfont Relative $L_2$ errors.\,
    \bfseries s, \normalfont Comparing metrics vs. permeability contrast.
    \bfseries Note:
    \normalfont\ Results in b--j are at 90 days. In f--j and o-s, NBM uses $N_b=1000$, and all methods (NBM, FVM, and vanilla PINN) employ a $50\times 50$ spatial resolution. Pressure is reported in $\mathrm{bar}$, velocity in $\mathrm{m\,day^{-1}}$, and concentration in $\mathrm{ppm}$. 
    }
  \label{fig:fig3}
\end{figure}
\vspace{-15pt}
\vspace{1.ex}
\subsection{NBM remains accurate and robust for multiscale permeability fields}
\vspace{-1.5ex}
For this Darcy--transport problem but under a challenging multiscale permeability field (Fig.~\ref{fig:fig3}a), increasing the neural basis size $N_b$ yields similar spectral-like residual decay together with systematic basis refinement (Fig.~\ref{fig:fig3}b--e).
Using the FVM solution as a baseline, fieldwise comparisons show that energy-consistent weighting is critical in the multiscale regime: without weighting, the flow-field energy-spectrum discrepancy $E(k)$ increases from $0.91\%$ to $7.05\%$ ($7.7\times$ increase), and the transport solution driven by this flow field deteriorates sharply, with the concentration relative $L_2$ difference rising from $4.09\%$ to $17.6\%$ ($4.3\times$) (Fig.~\ref{fig:fig3}f--j).
By contrast, a vanilla PINN fails to produce meaningful solutions.
Because a ground-truth is unavailable for this discrete-valued permeability field, the FVM solution is used only as a comparison baseline; discrepancies relative to it should not be interpreted as inaccuracy of NBM.
We therefore include a manufactured-solution benchmark with a separable multiscale permeability to assess multiscale accuracy and robustness.
NBM exhibits the same decay trend while converging to more accurate solutions than FVM (Fig.~\ref{fig:fig3}l--m), and the energy spectra confirm that it matches the dominant low-frequency mode of the true spectrum with much smaller spectral error ($0.31\%$ vs $8.95\%$) (Fig.~\ref{fig:fig3}n).
NBM's accuracy is also reflected in fieldwise errors: the relative $L_2$ errors drop from $0.078\%/4.02\%/6.74\%/4.52\%$ (FVM) to $0.052\%/0.61\%/1.03\%/0.51\%$ for $(p,u_x,u_y,c)$ (Fig.~\ref{fig:fig3}o--r).
Moreover, robustness to increasing permeability contrast is demonstrated up to $10^4$, with variable-wise and spectral relative $L_2$ errors and the K--S distance remaining controlled across the tested contrasts (Fig.~\ref{fig:fig3}s).
The experiment details of this example are provided in Appendices~\ref{Appx:F}, \ref{Appx:G} and \ref{Appx:I}.

\section{NBM operator learning for parametric Darcy flow--transport}
\vspace{-1.5ex}
The results above establish that NBM is accurate and robust, but single-instance NBM solver requires dense least-squares solutions. 
With $N_b$ bases and $M$ residual evaluations, the computational cost scales as $\mathcal{O}(M N_b^2)$, so NBM is not intended to compete with classical solvers in single-query settings. 
The situation changes in many-query regimes with parametric variations in PDE coefficients and boundary conditions, where repeated linear solves become the bottleneck in engineering design, control, and uncertainty analysis. 
NBM--OL addresses this by learning, offline, the parametric dependence of the solution coefficients in the predefined neural basis space and then evaluating them online at negligible cost.

We use NBM--OL to learn trajectory operators for the parametric system \eqref{eq:parametric_system}. Parametric encodings, together with time when relevant, are mapped to the neural basis coefficients, as summarized in Fig.~\ref{fig:fig4}. 
The parametric inputs include permeabilities $\boldsymbol{\xi}_{\kappa}$ and boundary values $\boldsymbol{\xi}_{b}$. 
To regularize learning and reduce dimensionality, we represent these inputs with low-dimensional encodings, as is standard in parametric PDE modeling \cite{quarteroni2015reduced}; for permeability, we use coarse blockwise parameters commonly adopted in geostatistical studies \cite{gavalas1976reservoir,oliver2011recent}. 
The learned map predicts the neural basis coefficients, interpreted as the reduced coordinates of the solution in the NBM space. 
If additional compression is
beneficial, we can apply proper orthogonal decompositio (POD) to the solution coefficient
trajectories to further reduce the operator output dimension. 
The training process can be fully self-supervised,
requiring no training data and thus avoiding dependence on expensive simulation databases.
In addition, it minimizes the same residuals \eqref{eq:nbm_darcy_picard} and \eqref{eq:nbm_transport_upwind} used by the single-instance NBM solver for flow and transport. 
Notably, each residual provides an interpretable certificate for monitoring, diagnosis, and improving the learning process.
Next, we first establish the respective equivalence relations with the solution errors using Lemmas~\ref{lem:darcy_equivalence} and~\ref{lem:transport_equivalence} and then show that NBM--OL breaks the fundamental bottleneck of repeated parametric solves in many-query settings.
\vspace{-1.5ex}
\begin{figure}[!ht]
    \centering
    \includegraphics[width=1\linewidth]{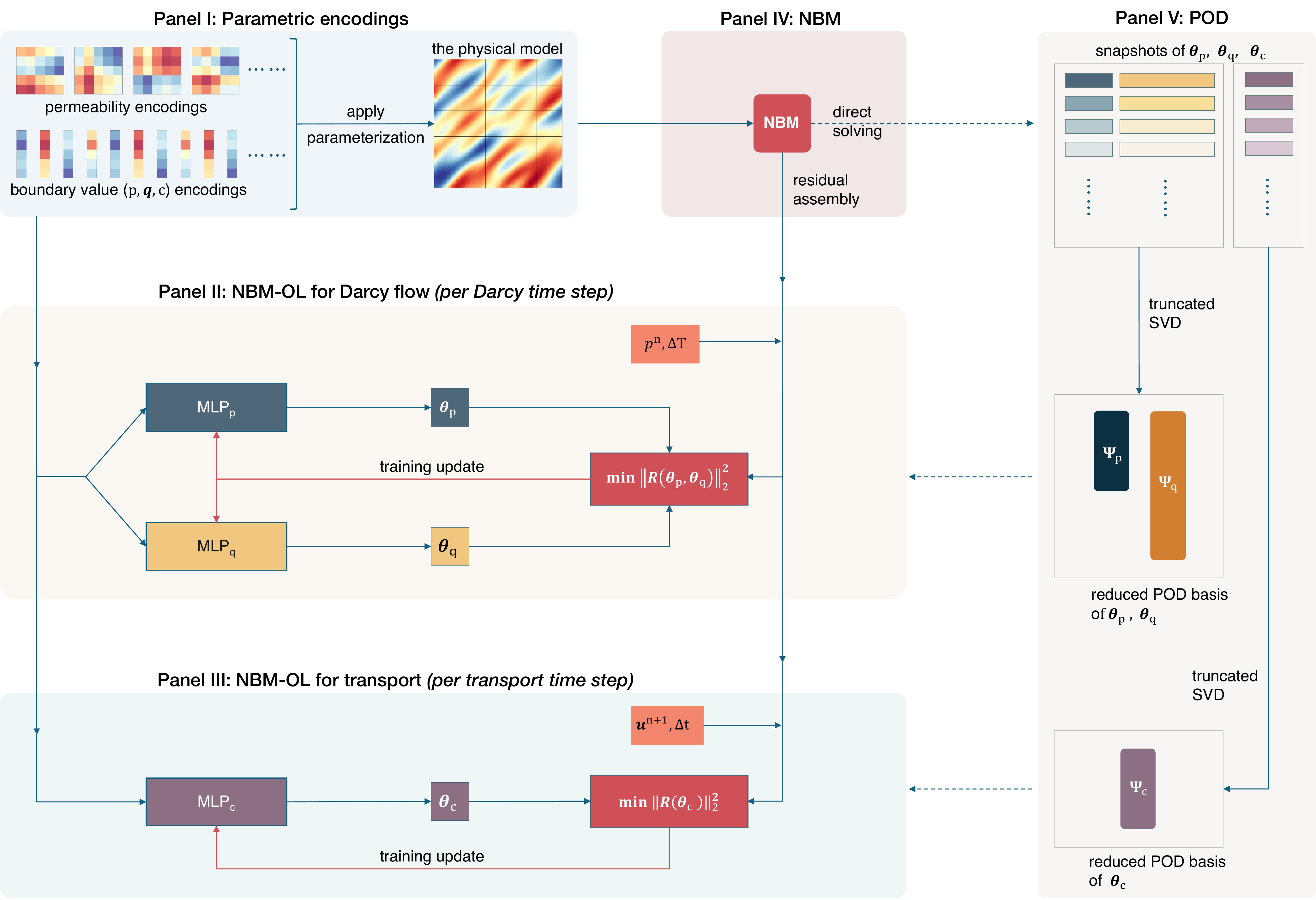}
    \caption{%
    \bfseries NBM-OL for parametric Darcy--transport systems.\,
    \bfseries Panel I,
    \normalfont Parameterization of permeability and boundary value. 
    \bfseries Panel II,
    \normalfont NBM--OL for Darcy flow (per Darcy time step, $\Delta T$).
    $\mathrm{MLP}_{p}$ and $\mathrm{MLP}_{q}$ learn the pressure and mass-flux solution coefficients $\boldsymbol{\theta}_{\mathrm{p}}$ and $\boldsymbol{\theta}_{\mathrm{q}}$.
    Training is self-supervised by $\min \lVert \boldsymbol{R}(\boldsymbol{\theta}_{\mathrm{p}},\boldsymbol{\theta}_{\mathrm{q}}) \rVert_{2}^{2}$ from \eqref{eq:nbm_darcy_picard}, where $\boldsymbol{R}(\theta_{\mathrm{p}},\theta_{\mathrm{q}})=\mathbf{W}\bigl(\mathbf{A}_{\mathrm{flow}}\,\boldsymbol{\theta}_{p,q}-\mathbf{b}_{\mathrm{flow}}\bigr)$ with time stepping indices omitted. 
    $p^n$ denotes the pressure from the previous Darcy time-step, used for the current update.
    \bfseries Panel III,
    \normalfont NBM--OL for transport (per transport time-step, $\Delta t$).
    $\mathrm{MLP}_{c}$ learns concentration coefficients $\boldsymbol{\theta}_{\mathrm{c}}$ by $\min \lVert \boldsymbol{R}(\boldsymbol{\theta}_{\mathrm{c}}) \rVert_{2}^{2}$ from \eqref{eq:nbm_transport_upwind}, where $\boldsymbol{R}(\theta_{\mathrm{c}})=\mathbf{A}_{\mathrm{up}}\,\boldsymbol{\theta}_{c}-\mathbf{b}_{\mathrm{up}}$ with time stepping indices omitted.
    $\mathbf{u}^{n+1}$ denotes the velocity after the current Darcy time-step update, computed as $\mathbf{u}^{n+1}=\mathbf{q}^{n+1}/\rho^{n+1}$.
    \bfseries Panel IV,
    \normalfont Optional POD compression of neural basis coefficients.
    $\Psi_{p}$, $\Psi_{q}$, and $\Psi_{c}$ are the respective reduced POD bases.
    }
    \label{fig:fig4}
\end{figure}

\subsection{Residual--error certificate for the weighted Darcy projection}
\vspace{-1.5ex}
We consider a single implicit time step of the slightly-compressible Darcy update $T^n\to T^{n+1}$.
The density increment satisfies
\(
\rho(p^{n+1})-\rho(p^n)=\rho_{\mathrm{0}}\,c_f\,(p^{n+1}-p^n).
\)
For clarity, here we show linearization by lagging the density prefactor at time level $n$ and introduce
\(
\alpha:=\rho(p^n)\frac{\kappa}{\mu}, \,\,
\beta:=\frac{\varepsilon\,\rho_{\mathrm{0}}\,c_f}{\Delta T}.
\)
Let $(p,\mathbf{q})$ denote the exact mixed state and $(\tilde p,\tilde{\mathbf{q}})$ its NBM approximation.
The corresponding errors are $e_p=\tilde p-p$ and $e_{\mathbf q}=\tilde{\mathbf q}-\mathbf q$.
We write \eqref{eq:momentum}--\eqref{eq:mass} in the linearized time step-dependent form
\begin{align}
\mathbf{q} + a\,\nabla p &= 0 \quad \text{in }\Omega, \label{eq:linear_momentum}\\
\nabla\cdot \mathbf{q}+ \beta\, p &= 0 \quad \text{in }\Omega, \label{eq:linear_mass}
\end{align}
with mixed boundary conditions \(p=p_D\) on \(\partial\Omega_p\) and \(\mathbf{q}\cdot \mathbf{n}=q_N\) on \(\partial\Omega_q\).
We define the residuals
\[
\begin{array}{ll}
r_m := \tilde{\mathbf q}+a\nabla \tilde p, \,\,\,
&
r_c := \nabla\!\cdot\tilde{\mathbf q}+\beta \tilde p,\\
r_D := \tilde p-p_D, \,\,\,
&
r_N := \tilde{\mathbf q}\cdot \mathbf{n}-q_N.
\end{array}
\]
and collect them in the weighted residual norm
\begin{equation*}
\|R(\tilde p,\tilde{\mathbf q})\|_{W}^{2}
:=
\|w_m\,r_m\|_{L^2(\Omega)}^{2}
+\|w_c\,r_c\|_{L^2(\Omega)}^{2}
+\|w_D\,r_D\|_{L^2(\partial\Omega_p)}^{2}
+\|w_N\,r_N\|_{L^2(\partial\Omega_q)}^{2},
\end{equation*}
where $w_{\star}$ is given in \eqref{eq:ls-weight} for $\star\in\{m,c,D,N\}$.
We further write the natural graph norm of error
\begin{equation*}
\|(e_p,e_{\mathbf q})\|_{\mathcal E}^2
:=
\|e_p\|_{H^1(\Omega)}^2
+\|e_{\mathbf q}\|_{H(\mathrm{div};\,\Omega)}^2
+\|e_p\|_{L^2(\partial\Omega_p)}^2
+\|e_{\mathbf q}\|_{L^2(\partial\Omega_q)}^2.
\end{equation*}
Then, Lemma~\ref{lem:darcy_equivalence} shows that this weighted residual norm is equivalent to the natural graph norm such that $\|R(\tilde p,\tilde{\mathbf q})\|_{W}\asymp\,\|(e_p,e_{\mathbf q})\|_{\mathcal E}$.
Consequently, we provide a computable residual--error certificate. 

\begin{lemma}[Residual--error equivalence for the weighted first-order least-squares Darcy problem]
\label{lem:darcy_equivalence}
Assume $0<a_{\min}\le a(x)\le a_{\max}$ and $0<\beta_{\min}\le \beta(x)\le \beta_{\max}$ a.e.\ in $\Omega$.
Assume the boundary split $\partial\Omega=\partial\Omega_p\cup\partial\Omega_q$ ensures uniqueness, e.g.,
$|\partial\Omega_p|>0$.
Then there exist constants $C_1,C_2>0$, independent of $(\tilde p,\tilde{\mathbf q})$, such that
\begin{equation}\label{eq:equivalency_norm}
C_1 \,\|(e_p,e_{\mathbf q})\|_{\mathcal E}
\le
\|R(\tilde p,\tilde{\mathbf q})\|_{W}
\le
C_2 \,\|(e_p,e_{\mathbf q})\|_{\mathcal E},
\end{equation}
\end{lemma}

\begin{proof}
The equivalence between $\|\cdot\|_{W}$ and the corresponding unweighted $L^2$-residual norm follows from
$0<c_\star\le w_\star\le C_\star$.
The error--residual equivalence \eqref{eq:equivalency_norm} is the $U$-ellipticity result for the
first-order system least-squares operator associated with \eqref{eq:linear_momentum}--\eqref{eq:linear_mass};
see, e.g., \cite{cai1994first,gantner2021further}.
\end{proof}
\vspace{-3.ex}

\subsection{Residual--error certificate for the stabilized transport projection}
\vspace{-1.5ex}
We next consider an implicit transport sub-timestep $t^m\to t^{m+1}$ based on the upwind control-volume balance in \eqref{eq:transport_fvm}. 
Let $c$ denote the exact concentration and $\tilde c$ its NBM approximation, and define the error by $e_c:=\tilde c-c$.
We further define the residual for this sub-timestep as $r_{\mathrm{up}}:=A_{\mathrm{up}}\tilde c-b$, so that $r_{\mathrm{up}}=A_{\mathrm{up}}e_c$. 
For an arbitrary cellwise scalar field $v=\{v_K\}$, we introduce the pore-volume scaled diagonal mass weight
\(
D:=\mathrm{diag}(\varepsilon |K|),\,
\|v\|_D^2:=\sum_K (\varepsilon|K|)\,v_K^2.
\)
For each interior face $f=K|L$, let $\mathbf{n}_{K,f}$ denote the unit normal on $f$ pointing outward from cell $K$, and set
\(
\alpha_f:=\big|(\mathbf{u}\cdot \mathbf{n}_{K,f})_f\big|\,|f|.
\)
We can write the upwind dissipation semi-norm as
\(
|v|_{\mathrm{up}}^2:=\frac12\sum_{f=K|L}\alpha_f\,(v_K-v_L)^2.
\)
We then define the transport energy norm
\(
\|v\|_{\mathcal E,\mathrm{tr}}^2:= \frac{1}{\Delta t}\|v\|_D^2 + |v|_{\mathrm{up}}^2.
\)
Since $|v|_{\mathrm{up}}^2\ge 0$, the transport energy norm controls the cellwise Euclidean error of concentration
\begin{equation*}
\|e_c\|_2
\ \le\
\sqrt{\frac{\Delta t}{\varepsilon |K|_{\min}}}\,\|e_c\|_{\mathcal E,\mathrm{tr}},
\qquad
|K|_{\min}:=\min_K |K|.
\end{equation*}
On a fixed discretization, $\|\cdot\|_{\mathcal E,\mathrm{tr}}$ and $\|\cdot\|_2$ are equivalent up to constants. The next lemma shows that the Euclidean residual norm $\|r_{\mathrm{up}}\|_2$ is equivalent to the transport energy norm of the error, i.e., $\|r_{\mathrm{up}}\|_2 \,\asymp\, \|e_c\|_{\mathcal E,\mathrm{tr}}$, thereby providing a computable residual--error certificate.

\begin{lemma}[Residual--error equivalence for upwind control-volume transport problem]
\label{lem:transport_equivalence}
Assume $\varepsilon>0$ and a fixed transport discretization (mesh, velocity field, and time step).
Then there exist constants $C_1,C_2>0$, independent of $\tilde c$, such that
\begin{equation}
C_1\,\|e_c\|_{\mathcal E,\mathrm{tr}}
\ \le\
\|r_\mathrm{up}\|_{2}
\ \le\
C_2\,\|e_c\|_{\mathcal E,\mathrm{tr}}.
\label{eq:transport_equivalence}
\end{equation}
\end{lemma}

\begin{proof}
For the upwind control-volume operator $A=\frac1{\Delta t}D+V$, the symmetric part
$V_s=\tfrac12(V+V^\top)$ induces the classical upwind jump dissipation, i.e.,
$v^\top V_s v = |v|_{\mathrm{up}}^2\ge 0$ for all cell vectors $v$; see, e.g., \cite{boyer2012analysis,doi:https://doi.org/10.1002/9781119176817.ecm2010}.
Hence $G:=\frac1{\Delta t}D+V_s\succ0$ and $\|v\|_{\mathcal E,\mathrm{tr}}^2=v^\top G v$.
Writing $A=G+V_k$ with skew-symmetric $T_k=\tfrac12(V-V^\top)$, the DPG theory yields the residual--error equivalence in the induced dual norm,
$\|e\|_G \asymp \|Ae\|_{G^{-1}}$; see, e.g., \cite{demkowicz2010class,demkowicz2011class,demkowicz2025discontinuous}.
Finally, on a fixed discretization, the Euclidean residual $\|r\|_2$ is equivalent to $\|r\|_{G^{-1}}$
by standard spectral norm equivalence for symmetric positive definite matrices, which completes \eqref{eq:transport_equivalence}.
\end{proof}
\vspace{-2ex}

\subsection{Parametric Darcy flow under varying permeability and boundary flux}
\vspace{-1.5ex}
We here assess whether NBM--OL can learn the Darcy flow operator under two different types of parametric variation: distributed permeability uncertainty and scalar boundary-flux control. 
Case-I considers an incompressible setting in which the base permeability field is parameterized through a piecewise-constant $5\times5$ coarse-block representation (block25), yielding a reduced operator family that retains the dominant spatial variability. 
Case-II considers the multiscale CO$_2$ storage model, where the varying parameter is the uniform top-boundary mass flux in $[488,\,976]~\mathrm{kg/(m^2\,day)}$. 
Together, these two cases probe whether NBM--OL can handle both field variation and control variation (Fig.~\ref{fig:fig5}).

\begin{figure}[hb!]
    \centering
    \includegraphics[width=1\linewidth]{figs/fig5.jpg}
    \caption{%
    \bfseries NBM--OL learns parametric Darcy flow operators under varying permeability and boundary flux.\,
    \bfseries a, \normalfont Problem setup (kv) and base permeability field.\,
    \bfseries b, \normalfont Representative block25 parametrizations.\,
    \bfseries c-d,
    \normalfont Left to right: true pressure/velocity magnitude, predicted pressure/velocity magnitude, respective absolute error fields.\,
    \bfseries e,
    \normalfont In-distribution generalization.\,
    \bfseries f-h,
    \normalfont Out-of-distribution generalization across heterogeneity level, correlation length, and spatial rotation.\,
    \bfseries i, \normalfont Problem setup (bv) and permeability field.\,
    \bfseries j-k,
    \normalfont Left to right: true pressure/velocity magnitude, predicted pressure/velocity magnitude, respective absolute error fields, at 200~days.\,
    \bfseries l,
    \normalfont In-distribution generalization.\,
    \bfseries m,
    \normalfont Out-of-distribution generalization in boundary flux.\,
    \bfseries n,
    \normalfont Out-of-distribution generalization in time horizon.\,
    \bfseries o, \normalfont Residual--error correlation during training: $\sqrt{\mathcal{E}_{\mathrm{rel}}}$ vs. relative $L_2$ errors.\,
    \bfseries p, \normalfont Training dynamics measured by normalized $\mathcal{E}_{\mathrm{rel}}$.
    \bfseries q, \normalfont Runtime comparison and speedups of NBM--OL relative to FVM, reported as CPU wall-time statistics in seconds.
    \bfseries Note:
    \normalfont\ Pressure is reported in $\mathrm{bar}$ and velocity in $\mathrm{m\,day^{-1}}$. 
    }
    \label{fig:fig5}
\end{figure}
For Case-I, NBM--OL achieves consistently low in-distribution relative $L_2$ errors for pressure and velocity (Fig.~\ref{fig:fig5}e). 
More importantly, the learned operator remains robust under out-of-distribution shifts in heterogeneity level, correlation length, and spatial rotation (Fig.~\ref{fig:fig5}f--h), indicating that the model captures structural flow responses rather than narrowly interpolating within the training ensemble.
For Case-II, the learned operator again produces tightly clustered sub-percent errors in-distribution (Fig.~\ref{fig:fig5}l). 
Its accuracy is further preserved when extrapolating the boundary-flux magnitude by $50\%$ beyond the training range and when extending the prediction horizon by $100\%$ (Fig.~\ref{fig:fig5}m--n), demonstrating stable parametric and temporal generalization in the transient setting.

The training diagnostics show a clean decay of the normalized residual and, more importantly, a near-linear relation between $\sqrt{\mathcal{E}_{\mathrm{rel}}}$ and the solution $L_2$ errors across test cases; see Fig.~\ref{fig:fig5}l. 
This identifies $\sqrt{\mathcal{E}_{\mathrm{rel}}}$ as a practically useful accuracy indicator during learning, unlike generic optimization losses that need not reflect state accuracy in operator learning. 
As expected, NBM--OL supports rapid repeated parametric evaluation, delivering orders-of-magnitude speedups over FVM (Fig.~\ref{fig:fig5}q). 
The experiment setups and implementation details are provided in Appendix~\ref{Appx:H}.

\vspace{-1.5ex}
\subsection{Parametric Darcy--transport under coupled flow and transport dynamics}
\vspace{-1.5ex}
In Fig.~\ref{fig:fig6}, we consider a more stringent task in which transport is coupled with Darcy flow. 
The transport component introduces sharper, hyperbolic solution features, so accurate prediction of evolving concentration fronts becomes a harder test for reduced operator representations. 
In this setting, both the boundary mass-flux and the injected tracer concentration are treated as parameters within the same multiscale CO$_2$ storage configuration introduced earlier.

\begin{figure}[hb!]
    \centering
    \includegraphics[width=1\linewidth]{figs/fig6.jpg}
    \caption{%
    \bfseries NBM--OL learns parametric Darcy–transport operators under varying boundary flux and concentration.\,
    \bfseries a, \normalfont Problem setup for boundary-flux-varying (bv) and injection-concentration-varying (cv) Darcy--transport operator learning.\,
    \bfseries b--c,
    \normalfont Left to right: true concentration, predicted concentration with the tracer front overlaid, and the corresponding absolute error; b is at 20~days and c is at 200~days.\,
    \bfseries d,
    \normalfont In-distribution generalization, measured by relative $L_2$ errors at 200~days over 100 random tests.\,
    \bfseries e,
    \normalfont Out-of-distribution generalization in boundary flux and concentration, measured by relative $L_2$ errors at 200~days over 100 random runs.\,
    \bfseries f,
    \normalfont Out-of-distribution generalization in time horizon, measured by relative $L_2$ errors at 400~days over 100 random runs.\,
    \bfseries g, \normalfont Training dynamics measured by normalized $\mathcal{E}_{\mathrm{rel}}$.
    \bfseries h, \normalfont Residual--error correlation during training: $\sqrt{\mathcal{E}_{\mathrm{rel}}}$ vs. relative $L_2$ errors.\,
    \bfseries i, \normalfont Runtime comparison and speedups of NBM--OL relative to FVM, reported as CPU wall-time statistics in seconds over 100 random runs.
    \bfseries Note:
    \normalfont\ Concentration is reported in $\mathrm{ppm}$.
    }
    \label{fig:fig6}
\end{figure}
As shown in Fig.~\ref{fig:fig6}b--c, NBM--OL accurately predicts the concentration fields while preserving the location and shape of the evolving tracer fronts. 
When the boundary flux and injected concentration are sampled within the training ranges, the resulting in-distribution errors remain tightly clustered at very low levels; see Fig.~\ref{fig:fig6}d. 
More importantly, when both parameter ranges are extrapolated by $50\%$ and the prediction horizon is extended by $100\%$, the error increase remains moderate and no instability or catastrophic degradation is observed (Fig.~\ref{fig:fig6}e--f). 
This indicates that the learned operator captures the coupled flow--transport structure and front evolution mechanisms, rather than merely interpolating within the training ensemble.

Figures~\ref{fig:fig6}g--h further show that this residual-based accuracy indicator remains effective in the coupled flow--transport setting. 
Even with sharper transport dynamics, $\sqrt{\mathcal{E}_{\mathrm{rel}}}$ continues to track concentration errors closely during training, indicating that the residual--error relation established for Darcy flow alone carries over to the more demanding coupled problem. 
Similarly, we achieve orders-of-magnitude speedups over FVM for parametric rollout using NBM--OL (Fig.~\ref{fig:fig6}i). 
See Appendix~\ref{Appx:H} for experiment setups and implementation details.

\section{\textbf{Conclusion}}
\vspace{-1.ex} 
NBM is a general, projection-based neural framework that unifies solving and learning PDE-based systems. 
Its central idea is to move physics out of heuristic loss-driven training and into the neural discretization itself through physics-conforming approximation spaces and operator-induced residual metrics. 
It yields a well-conditioned deterministic minimization, in which the resulting residual is a numerically meaningful certificate tied to both solution approximation and PDE enforcement.
This gives stability under basis enrichment, supports systematic accuracy assessment, and produces reduced coordinates as natural targets for parametric learning.  
These coordinates can be learned consistently across parameter instances, while the residual certificate remains an interpretable measure for monitoring, diagnosis, and improving operator predictivity.
Beyond the Darcy--transport instantiation, the new method applies whenever a PDE admits stable discrete enforcement under an appropriate residual norm.
Although we experiment rather simplified parameterizations, NBM--OL naturally extends to general parametric variations in operator coefficients, boundary conditions, and source terms, including spatially-varying and time-dependent profiles.
From an engineering perspective, this makes the framework particularly attractive for repeated-query modeling under varying material properties, boundary conditions, and operating controls.
\vspace{+1ex}

There are several directions to further broaden robustness and generality. 
Our current neural bases are global and can struggle with localized non-smoothness or nonconvex geometries, where Gibbs-type artifacts \cite{gottlieb1997gibbs} may cause loss of accuracy. 
NBM, however, is not limited to global representations and can be localized via domain decomposition, partition-of-unity, or spectral-element-style \cite{patera1984spectral} global--local hybrids. 
The presented weighting and stabilization are based on collocation and should be viewed as a simplified, not optimal, choice.
More rigorous variational constructions can draw on DPG-inspired designs.
Within this umbrella, we can further benefit from residual localization via Riesz representers \cite{demkowicz2012class}, yielding a computable error indicator that drives adaptive enrichment where additional attention is needed. 
In this way, NBM can further provide a diagnostic tool that guides modeling design.
Measurements can also be appended as additional residual blocks to uncover system parameters, informed by model-data identifiability and compatibility.
Finally, NBM is not intended to circumvent the curse-of-dimensionality in high-dimensional PDEs, where the number of basis functions can become impractial. 
Most PDEs in science and engineering, however, are posed on 1D-3D spatial domains, which is the regime NBM targets. Extending to fully 3D systems is thus another natural step, where scalable residual evaluation, randomized least-squares, and effective preconditioning warrant further study.

\appendix
\titleformat{\section}{\normalfont\Large\bfseries}{Appendix \Alph{section}:}{+0.5em}{}
\renewcommand{\thesubsection}{\Alph{section}.\arabic{subsection}}
\renewcommand{\theequation}{\Alph{section}.\arabic{equation}}
\renewcommand{\thefigure}{\Alph{section}.\arabic{figure}}
\renewcommand{\thetable}{\Alph{section}.\arabic{table}}
\counterwithin{equation}{section}
\counterwithin{figure}{section}
\counterwithin{table}{section}

\section{Coupled Darcy flow--transport}\label{Appx:A}
\vspace{-1.5ex}
We consider an isothermal compressible fluid flow through a two-dimensional heterogeneous porous continuum.
The fluid motion is described by Darcy’s law and coupled with mass conservation:
\begin{align}
\mathbf{q} + \rho \frac{\kappa}{\mu} \nabla p &= 0,
\label{eq:momentum} \\[3pt]
\frac{\partial (\varepsilon \rho)}{\partial t} + \nabla \cdot \mathbf{q} &= 0,
\label{eq:mass}
\end{align}
where $\mathbf{q}(\mathbf{x},t)$ denotes the mass-flux,
$p(\mathbf{x},t)$ is the pressure,
$\varepsilon(\mathbf{x})$ is the porosity,
and $\rho(p)$ is the density, which is expressed through an equation of state,
$\rho(p)=\rho_\mathrm{0}\big(1+c_f(p-p_\mathrm{0})\big)$. 
$c_f$ is the fluid compressibility and $\rho_\mathrm{0}$ is the reference density at reference pressure $p_\mathrm{0}$.
The viscosity $\mu$ and porosity $\varepsilon$ are assumed constant, while
$\kappa(\mathbf{x})$ is a spatially varying permeability field reflecting the heterogeneity of the medium.
The Darcy velocity is $\mathbf{u} = \mathbf{q}/\rho$.
In addition, we consider a passive scalar field $c(\mathbf{x},t)$ representing a transported concentration, which evolves according to
\begin{align}
\frac{\partial (\varepsilon c)}{\partial t} + \nabla \cdot (c\,\mathbf{u}) = 0,
\label{eq:transport}
\end{align}
and is assumed not to influence the flow field.
For clarity, source terms are omitted.
With appropriate boundary conditions, equations~\eqref{eq:momentum}--\eqref{eq:transport} define a coupled Darcy flow--transport system that can solved sequentially.
Here, we use a standard finite-volume discretization to generate reference solutions, employing backward Euler time integration and first-order upwinding scheme for advection.
In a compact form, we can express this dynamics as
\begin{equation}
\label{eq:pde_system}
\begin{aligned}
\mathcal{L}\bigl(s(\mathbf{x},t)\bigr) &= 0
\quad \text{in } \Omega, \\[3pt]
\mathcal{B}\bigl(s(\mathbf{x},t)\bigr) &= g
\quad \text{on } \partial\Omega,
\end{aligned}
\end{equation}
where $s(\mathbf{x},t)$ denotes the collection of components from the primary unknowns, namely the pressure $p$, the mass-flux $q$, and the concentration $c$.
In addition to single-instance solving, we consider a parametric family of problems in which the governing operators depend on parameters.
Specifically, we distinguish between parameters associated with the interior operator in terms of operator coefficients and those associated with the boundary operator,
\begin{equation}
\label{eq:parametric_system}
\mathcal{L}_{\boldsymbol{\xi}_\kappa}\bigl(s(\mathbf{x},t)\bigr) = 0
\quad \text{in } \Omega,
\qquad
\mathcal{B}_{\boldsymbol{\xi}_b}\bigl(s(\mathbf{x},t)\bigr) = g_{\boldsymbol{\xi}_b}
\quad \text{on } \partial\Omega,
\end{equation}
where $\boldsymbol{\xi}_\kappa$ parameterizes the spatial variations in the permeability field and
$\boldsymbol{\xi}_b$ parameterizes the variations in boundary values.
We collectively denote the full parameter vector by
$\boldsymbol{\xi} = (\boldsymbol{\xi}_\kappa, \boldsymbol{\xi}_b)$.
Equation \eqref{eq:parametric_system} defines a family of PDE systems sharing a common structure for NBM-powered operator learning.

\section{Neural basis construction and evaluation}\label{Appx:B}
\vspace{-1.5ex}
\textbf{Neural basis definition.}
Let
\(
X\in\mathbb{R}^{M\times d}
\)
collect $M$ prescribed spatial points in a $d$-dimensional domain, with each row giving one point. 
We build the neural basis by passing $X$ through a fixed-parameter residual network with $L$ layers and widths $\{p_1,\dots,p_L\}$, with $p_0=d$. For $\ell=0,1,\dots,L$, let
\(
A_\ell(X)\in\mathbb{R}^{M\times p_\ell}
\)
denote the layer-$\ell$ outputs evaluated at $X$, where rows correspond to spatial points and columns to basis functions. For $\ell=1,\dots,L$, we define
\[
\begin{array}{rcl@{\qquad}rcl}
t_\ell(X) & = & A_{\ell-1}(X)\,W_\ell^T + \mathbf{1}\,b_\ell^T,
&
a_\ell(X) & = & t_\ell(X)\odot \alpha_\ell^T,\\[0.3ex]
z_\ell(X) & = & \sigma\!\left(a_\ell(X)\right),
&
s_\ell(X) & = & \beta_\ell\,A_{\ell-1}(X)\,P_\ell^T,\\[0.3ex]
A_\ell(X) & = & s_\ell(X) + z_\ell(X),
\end{array}
\]
where $W_\ell,P_\ell\in\mathbb{R}^{p_\ell\times p_{\ell-1}}$, $b_\ell\in\mathbb{R}^{p_\ell}$, $\alpha_\ell\in\mathbb{R}^{p_\ell}$, and $\beta_\ell\in\mathbb{R}$. Here $\mathbf{1}\in\mathbb{R}^M$ is the all-ones vector, $\odot$ denotes elementwise multiplication with broadcasting, and $\sigma$ acts elementwise. Here, $\sigma(\cdot)$ is taken as $\tanh(\cdot)$. The columns of $A_L(X)$ define the neural basis at the $M$ spatial points. When needed, the basis can be enriched by concatenating selected intermediate-layer outputs with $A_L(X)$.

\textbf{Layer concatenation and constant basis.}
The neural basis in NBM need not be taken from a single layer. Given a user-specified index set $\mathcal{I}=\{i_1,\ldots,i_K\}$ with $0 \le i_k \le L$, we define
\[
\Phi(X)
=
\big[
A_{i_1}(X)\;\;
A_{i_2}(X)\;\;
\cdots\;\;
A_{i_K}(X)
\big]
\in
\mathbb{R}^{M\times P},
\qquad
P=\sum_{k=1}^K p_{i_k},
\]
where each $A_{i_k}(X)\in\mathbb{R}^{M\times p_{i_k}}$ denotes the corresponding layer output evaluated at the $M$ spatial points in $X$. This allows the basis to combine features from multiple depths. When needed, we further append a column of ones to include a constant basis function, so that spatially uniform components are represented explicitly.

\textbf{Analytic derivatives for PDE operators.}
PDE assembly requires spatial derivatives of the neural basis. Here, for each layer $\ell=0,\ldots,L$, we denote the Jacobian and Hessian by
\[
J_\ell(X)\in\mathbb{R}^{M\times p_\ell\times d},
\qquad
H_\ell(X)\in\mathbb{R}^{M\times p_\ell\times d\times d},
\]
where $J_\ell(X)_{m,:,i}$ and $H_\ell(X)_{m,:,i,j}$ store the first- and second-order derivatives of the layer-$\ell$ outputs at $x_m$. 
These quantities are propagated analytically through the residual network. For the nonlinear branch $z_\ell(X)=\sigma(a_\ell(X))$, we use
\[
\frac{\partial z_\ell}{\partial t_\ell}
=
\sigma'(a_\ell)\odot \alpha_\ell^T,
\qquad
\frac{\partial^2 z_\ell}{\partial t_\ell^2}
=
\sigma''(a_\ell)\odot (\alpha_\ell^2)^T,
\]
with all operations acting elementwise. In the first layer, the skip branch contributes a constant Jacobian and zero Hessian; for deeper layers, $J_{\ell-1}(X)$ and $H_{\ell-1}(X)$ are propagated through the linear and nonlinear branches by the chain rule. This yields analytic gradients, divergences, and Laplacians at the points in $X$, without numerical or automatic differentiation.

\textbf{Neural basis initialization.}
The fixed neural basis is constructed by explicitly initializing the parameters $\{W_\ell,b_\ell,P_\ell\}_{\ell=1}^L$. For layers $\ell\ge 2$, the entries of $W_\ell$ are sampled independently from $\mathcal N(0,1)$, the entries of $b_\ell$ independently from the uniform distribution on $[-r,r]$, where $r>0$ is a prescribed radius parameter, and $P_\ell$ is taken as the identity when $p_\ell=p_{\ell-1}$ and otherwise sampled from a zero-mean Gaussian distribution with variance scaled by $1/\sqrt{p_{\ell-1}}$. The first layer is treated separately to reduce geometric redundancy, since each neuron defines an affine function whose zero level set is a hyperplane and naive random initialization can produce nearly coincident or strongly correlated features. To avoid this, candidate first-layer neurons are filtered by simple geometric diversity criteria; in two dimensions, we exclude lines with negligible intersection length and lines that are nearly parallel to or spatially clustered with previously accepted ones.

\textbf{First-layer scaling.}
To make the first-layer basis independent of the size and location of the physical domain \(\Omega=[\underline{x}_1,\overline{x}_1]\times\cdots\times[\underline{x}_d,\overline{x}_d]\subset\mathbb{R}^d\), we define it with respect to the normalized reference domain $[-1,1]^d$. Let \(s_i=(\overline{x}_i-\underline{x}_i)/2\) and \(c_i=(\overline{x}_i+\underline{x}_i)/2\), so that \(\hat x_i=(x_i-c_i)/s_i\in[-1,1]\). Rather than rescaling coordinates explicitly, we absorb this affine map into the first-layer parameters: if \(W_{1,0}\), \(P_{1,0}\), and \(b_{1,0}\) are defined on the reference domain, then \(W_1 = W_{1,0}\,\mathrm{diag}(s)^{-1}\), \(P_1 = P_{1,0}\,\mathrm{diag}(s)^{-1}\), and \(b_1 = b_{1,0} - W_{1,0}(c\oslash s)\), with \(\oslash\) denoting elementwise division. Since the skip branch is linear in the input coordinates, this rescaling also induces a constant shift, so we evaluate \(s_1(X)=\beta_1(XP_1^T)+\mathbf{1}\gamma_1^T\), with \(\gamma_1=-\beta_1 P_{1,0}(c\oslash s)\). For the axis-aligned rectangular domains considered here, this mapping is exact, ensuring that the first-layer basis is defined consistently on $[-1,1]^d$, while no further domain-dependent scaling is needed in deeper layers.

\section{Geometric interpretation and conditioning effect of dual-layer neural basis}\label{Appx:C}
\vspace{-1.5ex}
\textbf{Geometric interpretation.}
Single-layer neural bases admit a geometric interpretation through the induced hyperplanes \cite{xu2025weak,zhang2023transnet,zhang2024transferable}. 
A random ensemble of such hyperplanes produces a soft partition of the domain, with each basis transitioning rapidly but smoothly across its associated hyperplane. 
However, random initialization can generate near-duplicate hyperplanes, leading to highly correlated features and ill-conditioned projection systems. 
To mitigate this while retaining the underlying hyperplane-induced geometry, we introduce a second linear--nonlinear map that recombines the first-layer features. 
Specifically, for each second-layer node $k=1,\dots,N_2$, we define
\(
z_k(\mathbf{x})
=
\sum_{j=1}^{N_1} W^{(2)}_{kj}\,\phi^{(1)}_j(\mathbf{x}) + b^{(2)}_k
\)
and
\(
\phi^{(2)}_k(\mathbf{x})=\sigma\!\big(z_k(\mathbf{x})\big).
\)
Thus, $z_k(\mathbf{x})$ is a fixed superposition of first-layer hyperplane responses, whose level sets encode a composite soft partition of the domain, while $\phi_k^{(2)}(\mathbf{x})$ is a bounded nonlinear transformation of this aggregated signal. 
In this way, the second layer preserves the geometric organization inherited from the first layer while reducing collinearity and improving conditioning. 
The final PDE solution is then represented as a linear combination of $\{\phi_k^{(2)}\}$. 
In our implementation, the rows of $\mathbf W^{(2)}\in\mathbb{R}^{N_2\times N_1}$ are sampled from $\mathcal N(0,1)$ and $\mathbf b^{(2)}\in\mathbb{R}^{N_2}$ from $[-r,r]$, after which all parameters are kept fixed. 
Fig.~\ref{fig:C1} illustrates this construction through first-layer features, second-layer pre-activation fields, and the corresponding dual-layer architecture.
Red connections highlight how first-layer features are superposed to form $z_k$, and applying the activation produces the corresponding post-activation basis function
$\phi_k^{(2)}(\mathbf{x})=\sigma(z_k(\mathbf{x}))$
that enters the final expansion.
This dual-layer architecture with the fixed linear--nonlinear recombination yields dramatically improved conditioning in the resulting projection systems compared with using $\{\phi_j^{(1)}\}$ directly, while retaining a clear geometric link to the first-layer hyperplane partition.

\textbf{Conditioning effect.}
Depth in the neural basis has a pronounced effect on the conditioning of the neural projection system as the basis size $N_b$ increases. Using the same setup of Fig.~\ref{fig:fig3}a, we compare single-layer, dual-layer, and triple-layer cases. Fig.~\ref{fig:C2} shows that the single-layer case becomes rapidly ill-conditioned, with the condition number reaching $3.34\times 10^{13}$ at $N_b=1000$, whereas adding one additional layer reduces it to $5.76\times 10^{8}$ at the same basis size, and a triple-layer basis further reduces it to $9.25\times 10^{6}$ at $N_b=2000$. 
For the dual-layer and triple-layer cases, all hidden layers are chosen to have the same width so that the comparison isolates the effect of depth rather than incidental width imbalance or bottleneck effects. Under this controlled setting, the improved conditioning is primarily attributable to the reduced basis correlation induced by multi-layer 
recombination. 
A complementary geometric view also suggests that, for the low-dimensional physical domains relevant to most engineering PDEs, raw partition count is unlikely to be the main bottleneck at basis sizes on $O(10^3)$.
With $N_1$ hyperplanes, the maximal number of induced partitions in $\mathbb{R}^d$ is \(P_d(N_1)=\sum_{k=0}^{d}\binom{N_1}{k}\). 
For example, $N_1=1000$ already yields $1001$, $500{,}501$, and $166{,}667{,}501$ partitions in $d=1,2,3$, respectively.

\begin{figure}[!t]
    \centering
    \includegraphics[width=1\linewidth]{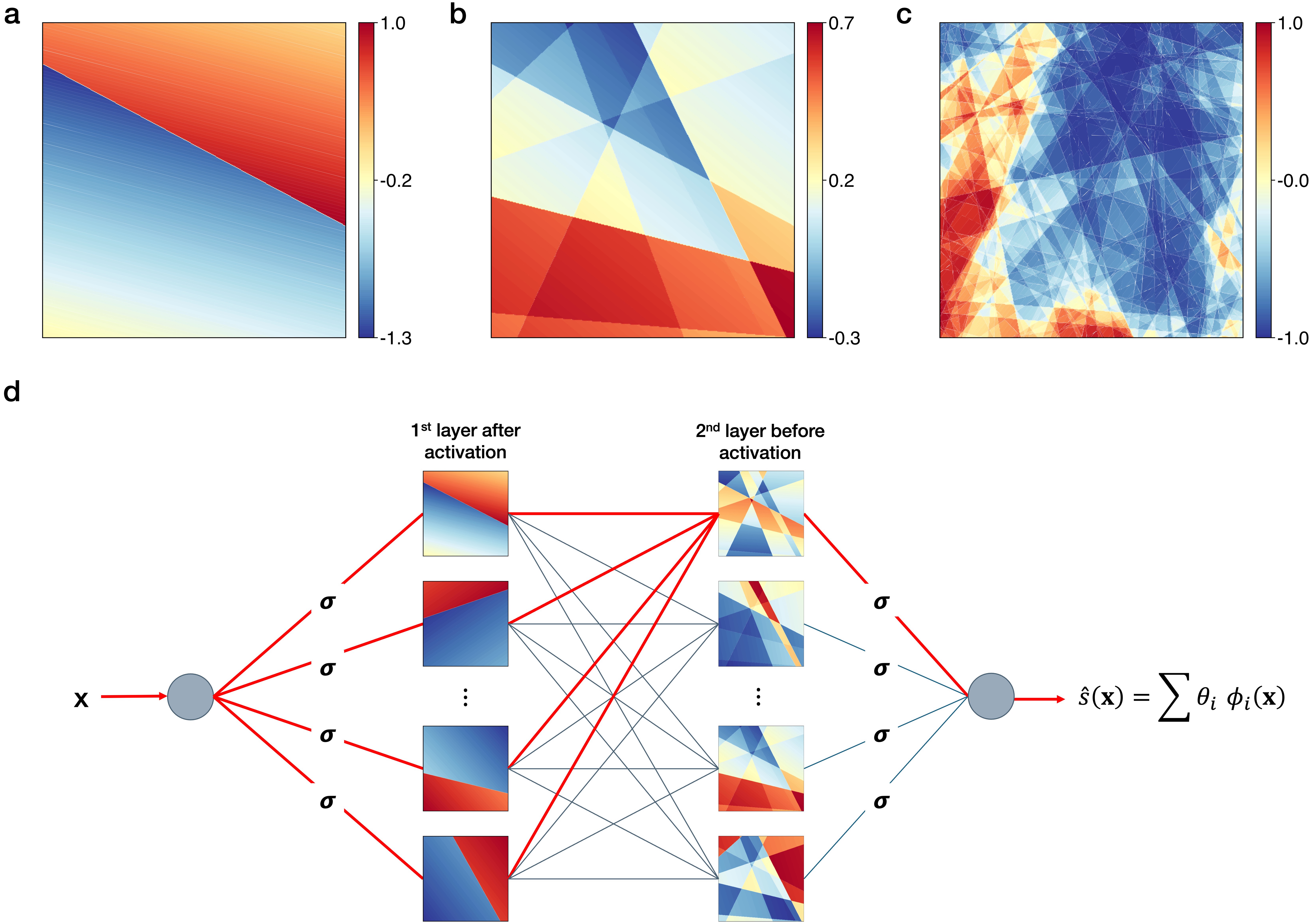}
    \caption{%
        \bfseries Geometric illustration of dual-layer neural basis.\,
        \bfseries a, \normalfont A first-layer hyperplane after activation.\,
        \bfseries b, \normalfont A second-layer pre-activation field formed by superposing $N_1=10$ first-layer features.\,
        \bfseries c, \normalfont A second-layer pre-activation field formed with $N_1=200$ first-layer features.\,
        \bfseries d, \normalfont Schematic of dual-layer basis generator: input produces first-layer hyperplanes, superpositions form pre-activations fields, and post-activation outputs the final basis functions.%
    }
    \label{fig:C1}
\end{figure}
\begin{figure}[!b]
    \centering
    \includegraphics[width=0.65\linewidth]{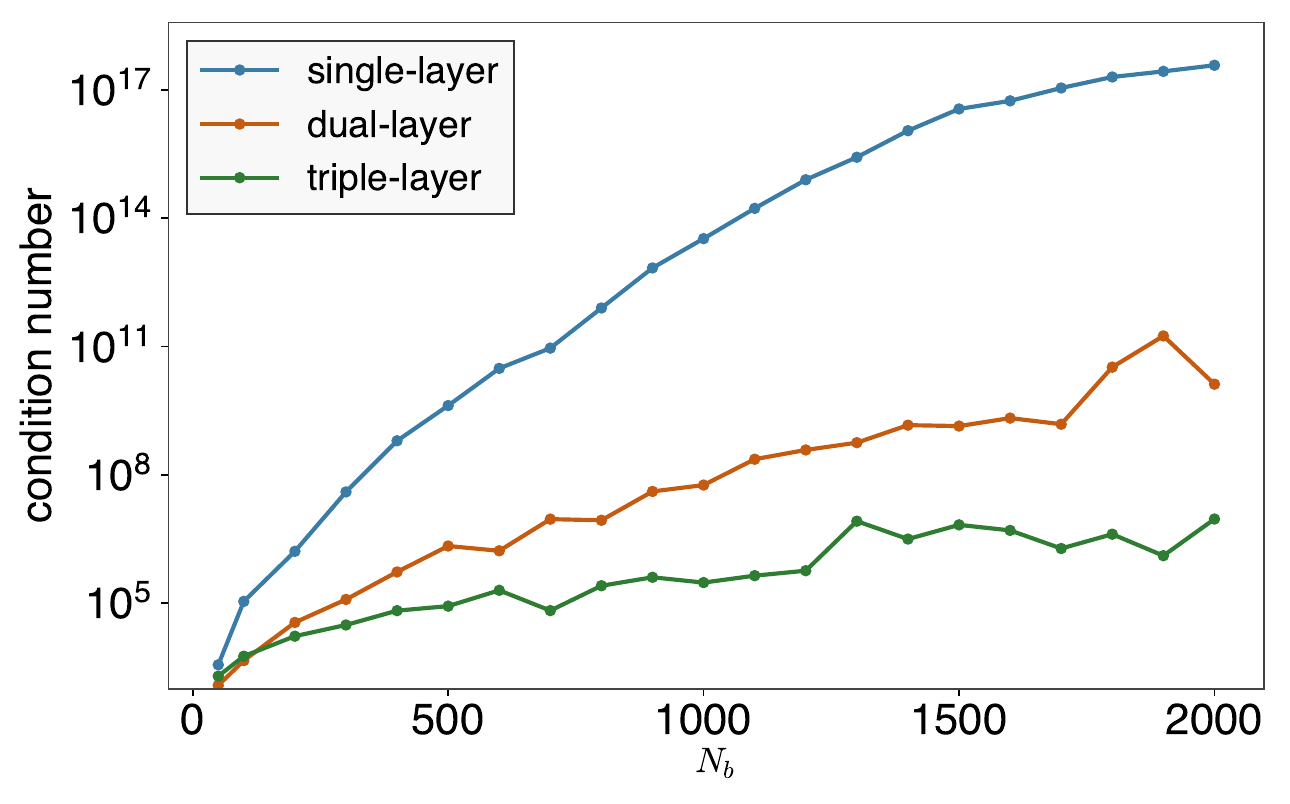}
    \caption{%
        \bfseries Effect of depth in neural basis generation on numerical conditioning.\,
        \normalfont Condition number of the assembled neural projection system versus neural basis size $N_b$ for single-layer, dual-layer, and triple-layer basis generators (logarithmic $y$-axis).%
    }
    \label{fig:C2}
\end{figure} 

\section{Expressivity and stability of neural basis space}\label{Appx:D}
\vspace{-1.5ex} 
Here, we assess how a given neural basis space represents the target pressure, mass-flux, and concentration fields. We use the setup shown in Fig.~\ref{fig:fig3}a and take the corresponding solution snapshots at day 10 as the reference fields.

\textbf{Least-squares projection and best-approximation error.}
We assess neural basis expressivity by the smallest discrete relative $L_2$ error attainable in approximating the reference fields. For a target scalar field $s(\mathbf{x})$ (pressure, a mass-flux component, or concentration) sampled at $M$ collocation points, we build the basis evaluation matrix $\Phi\in\mathbb{R}^{M\times N_b}$ with $\Phi_{ji}=\phi_i(\mathbf{x}_j)$ and perform a projection problem: 
\(\boldsymbol{\theta}^\star
=
\arg\min_{\boldsymbol{\theta}\in\mathbb{R}^{N_b}}
\left\|
\Phi\,\boldsymbol{\theta}-\mathbf{s}_{\mathrm{ref}}
\right\|_2^2
\)
and
\(
\widehat{s}(\mathbf{x})=\sum_{i=1}^{N_b}\theta_i^\star\,\phi_i(\mathbf{x}).
\)
Here $\mathbf{s}_{\mathrm{ref}}\in\mathbb{R}^M$ collects the reference values $s_{\mathrm{ref}}(\mathbf{x}_j)$. The resulting $\|\widehat{s}-s_{\mathrm{ref}}\|_2/\|s_{\mathrm{ref}}\|_2$ is thus the minimum discrete relative $L_2$ error achievable under the chosen sampling and norm. 
For the mass-flux $\mathbf{q}=(q_x,q_y)$, we evaluate expressivity componentwise. 
When no Helmholtz decomposition is used, we apply the projection separately to $q_x$ and $q_y$. 
When Helmholtz decomposition is used, we first write the reference flux as $\mathbf{q}_{\mathrm{ref}}=\mathbf{q}_{\mathrm{ref}}^{(\mathrm{div})}+\mathbf{q}_{\mathrm{ref}}^{(\mathrm{curl})}$, then project the four scalar components $q_{x,\mathrm{ref}}^{(\mathrm{div})}$, $q_{y,\mathrm{ref}}^{(\mathrm{div})}$, $q_{x,\mathrm{ref}}^{(\mathrm{curl})}$, and $q_{y,\mathrm{ref}}^{(\mathrm{curl})}$ in their corresponding neural basis spaces, and finally reconstruct $\widehat{\mathbf q}$ by summation.
We report relative $L_2$ projection errors for $q_x$, $q_y$, and $|\mathbf q|$, together with the associated condition numbers of the respective projections.

\textbf{Pressure projection: dual-layer versus single-layer bases.}
Because pressure is smooth and dominated by low-frequency content, both single- and dual-layer bases approximate the reference field very well. The reference pressure $p_{\mathrm{ref}}$ is shown in Fig.~\ref{fig:D1}a, and the corresponding absolute projection errors $|\widehat{p}-p_{\mathrm{ref}}|$ for the dual-layer and single-layer bases are shown in Fig.~\ref{fig:D1}b--c. 
As summarized in Fig.~\ref{fig:D1}d, replacing the dual-layer basis with the single-layer basis increases the best-approximation error only slightly, from $0.001\%$ to $0.0011\%$, but increases the condition number much more severely, from $2.6\times 10^8$ to $1.7\times 10^{15}$. The main difference is therefore not approximation quality but numerical conditioning, and in the PDE-constrained least-squares projection such ill-conditioning can amplify numerical noise and destabilize the solution coefficient updates.

\begin{figure}[!b]
    \centering
    \includegraphics[width=1\linewidth]{figs/fig_d1.jpg}
    \caption{%
    \bfseries Neural basis expressivity and stability for pressure.\, 
    \bfseries a, \normalfont FVM reference pressure $p$ (bar).\, 
    \bfseries b-c, \normalfont Absolute projection error using dual-layer and single-layer scalar neural basis, respectively \,
    \bfseries d, \normalfont Relative $L_2$ projection error and condition number of the projection system.
    \bfseries Note:
    \normalfont\ Pressure is reported in $\mathrm{bar}$.
    }
    \label{fig:D1}
\end{figure}

\textbf{Mass-flux projection: dual-layer versus single-layer bases, with and without Helmholtz split.}
The mass-flux projection highlights two practically important effects: the Helmholtz split improves best-approximation, while the single-layer basis can substantially worsen conditioning. The reference magnitude $|\mathbf{q}_{\mathrm{ref}}|$ and the componentwise absolute projection errors for $q_x$ and $q_y$ are shown in Fig.~\ref{fig:D2}a--e, and the relative $L_2$ errors together with the condition numbers are summarized in Fig.~\ref{fig:D2}f. Within each basis generator, introducing the Helmholtz split reduces the projection error. For the dual-layer basis, the relative errors decrease from $1.00\%$ to $0.86\%$ for $q_x$, from $1.66\%$ to $1.21\%$ for $q_y$, and from $1.10\%$ to $0.80\%$ for $|\mathbf{q}|$; for the single-layer basis, they decrease from $1.29\%$ to $1.01\%$, from $1.92\%$ to $1.48\%$, and from $1.35\%$ to $1.06\%$, respectively. With the Helmholtz split, replacing the dual-layer basis by the single-layer basis still increases the best-approximation error, from $0.86\%$ to $1.01\%$ for $q_x$, from $1.21\%$ to $1.48\%$ for $q_y$, and from $0.80\%$ to $1.06\%$ for $|\mathbf{q}|$, while conditioning deteriorates much more severely: for $q_x$, the div-free block increases from $2.6\times 10^8$ to $1.2\times 10^{13}$ and the curl-free block from $6.3\times 10^6$ to $8.0\times 10^{12}$, with similar behavior for $q_y$. Without the Helmholtz split, replacing the dual-layer basis by the single-layer basis increases the relative errors from $1.00\%$ to $1.29\%$ for $q_x$, from $1.66\%$ to $1.92\%$ for $q_y$, and from $1.10\%$ to $1.35\%$ for $|\mathbf{q}|$, while the condition number rises from $2.6\times 10^8$ to $1.7\times 10^{15}$. Compared with pressure, the mass-flux errors are a few orders of magnitude larger, which is expected because $\mathbf q$ is less smooth: it depends on pressure gradients and multiscale permeability, both of which amplify high-frequency content. This is also reflected in Fig.~\ref{fig:D2}g, where the projected and reference spectra agree well over the dominant low-to-intermediate wavenumbers, while the remaining discrepancy is concentrated in the high-wavenumber tail.

\begin{figure}[!h]
    \centering
    \includegraphics[width=1\linewidth]{figs/fig_d2.jpg}
    \caption{%
    \bfseries Neural basis expressivity and stability for mass-flux.\, 
    \bfseries a, \normalfont FVM reference (kg\,m$^{-2}$\,s$^{-1}$).\, 
    \bfseries b-e, \normalfont Absolute projection error using dual-layer neural vector basis w/ Helmholtz decomposition, dual-layer neural vector basis w/o Helmholtz decomposition, single-layer neural vector basis w/ Helmholtz decomposition, and single-layer neural vector basis w/o Helmholtz decomposition, respectively.\,
    \bfseries f, \normalfont Relative $L_2$ projection errors for $q_x$, $q_y$, and $|\mathbf{q}|$ and condition numbers of the corresponding projection systems. \,
    \bfseries g, \normalfont Energy spectra $E(k)$ of the reference and the best projected mass-flux fields.
    \bfseries Note:
    \normalfont\ Mass-flux is reported in $\mathrm{kg~day^{-1}~m^{-2}}$.
    }
    \label{fig:D2}
\end{figure}

\textbf{Concentration projection: dual-layer versus single-layer bases and the effect of stabilization.}
Concentration projection behaves differently from pressure and mass flux because sharp advective fronts limit the best-approximation power of smooth neural bases.
However, stabilization can make the target more compatible with the smooth neural bases and reduce projection error markedly. 
The reference concentration, the corresponding projection errors of the dual-layer and single-layer bases, and the error--conditioning summary are shown in Fig.~\ref{fig:D3}a--d. Replacing the dual-layer basis by the single-layer basis increases the best-approximation error only slightly, from $0.92\%$ to $0.99\%$, but increases the condition number severely, from $2.6\times 10^8$ to $1.7\times 10^{15}$. 
The larger approximation difficulty comes instead from the limited regularity of advective concentration fields: direct projection of a sharp   produces substantially larger error, as shown in Fig.~\ref{fig:D3}e,g,i,k, because such interfaces are difficult for a smooth neural expansion to represent. 
In the NBM transport discretization, first-order upwinding introduces numerical stabilization by creating a narrow grid-scale transition zone, thereby replacing the discontinuous step by a profile with finite transition thickness. 
This makes the target much more compatible with the neural basis, and accordingly a five-cell smeared front yields a much smaller projection error, as shown in Fig.~\ref{fig:D3}f,h,j,k. The remaining discrepancy is localized near the stabilized interface. Fig.~\ref{fig:D3}k also highlights that the sharp-front and five-cell smeared profiles appear visually similar, yet their relative $L_2$ difference is $9.19\%$, while the best-approximation error for projecting the discontinuous front is $7.73\%$. Overall, these results show that transport benefits substantially from stabilization and that the neural basis represents the stabilized transport profile with low error.

\begin{figure}[!b]
    \centering
    \includegraphics[width=1\linewidth]{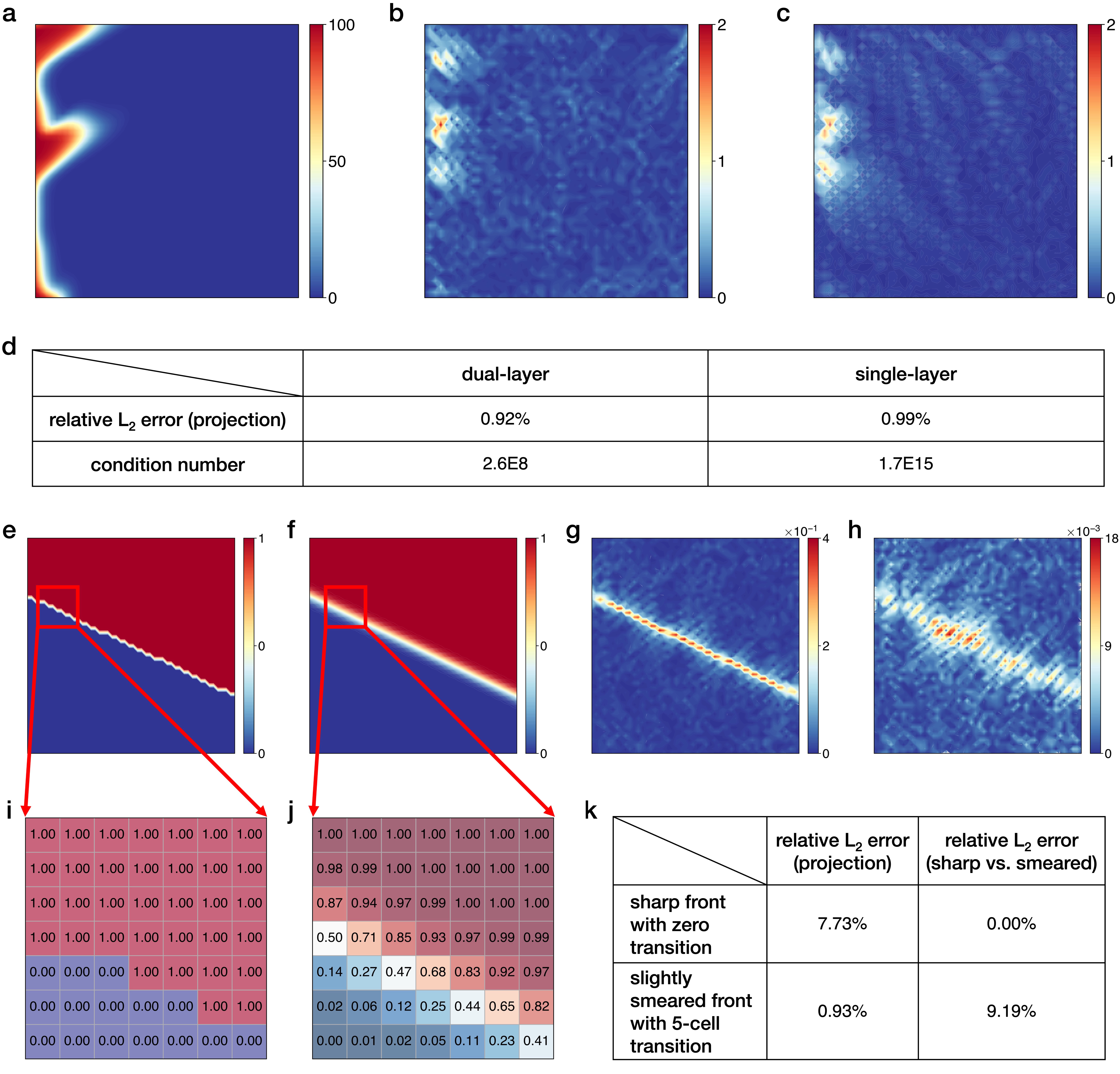}
    \caption{%
    \bfseries Neural-basis expressivity and stability for concentration.\, 
    \bfseries a, \normalfont FVM reference.\,
    \bfseries b-c, \normalfont Absolute projection error using dual-layer and single neural scalar basis, respectively.\,
    \bfseries d, \normalfont Relative $L_2$ projection error and condition number of the projection system.\,
    \bfseries e, \normalfont Sharp-front field with zero transition thickness (discontinuous step).\,
    \bfseries f, \normalfont Slightly smeared front with a 5-cell transition (stabilized profile).\,
    \bfseries g-h, \normalfont Absolute projection error for sharp-front and smeared-front cases, respectively.\,
    \bfseries i-j, \normalfont Cell values near sharp-front and smeared-front, respectively.\,
    \bfseries k, \normalfont Relative $L_2$ projection errors for the sharp-front and smeared-front fields, and the relative $L_2$ difference between the two fields (sharp vs.\ smeared).
    \bfseries Note:
    \normalfont\ Concentration is reported in $\mathrm{ppm}$.
    }
    \label{fig:D3}
\end{figure}

\section{Construction and effect of physics-conforming neural vector bases}\label{Appx:E}
\vspace{-1.5ex}
\textbf{Construction.}
Given scalar neural basis functions $\{\phi_i(\mathbf{x})\}$, we construct curl-free and divergence-free neural vector bases by applying differential operators directly to the scalar basis. In 2D, we define
\[
\boldsymbol{\phi}_i^{(\mathrm{curl})}(\mathbf{x})
=
\nabla \phi_i(\mathbf{x})
=
\begin{bmatrix}
\partial_x \phi_i(\mathbf{x})\\[2pt]
\partial_y \phi_i(\mathbf{x})
\end{bmatrix},
\qquad
\boldsymbol{\phi}_i^{(\mathrm{div})}(\mathbf{x})
=
\nabla^\perp \phi_i(\mathbf{x})
=
\begin{bmatrix}
\partial_y \phi_i(\mathbf{x})\\[2pt]
-\partial_x \phi_i(\mathbf{x})
\end{bmatrix},
\]
so that $\nabla\times \boldsymbol{\phi}_i^{(\mathrm{curl})}\equiv 0$ and $\nabla\cdot \boldsymbol{\phi}_i^{(\mathrm{div})}\equiv 0$. The resulting mass-flux representation takes the form
\[
q_{x,NN}(\mathbf{x})
=
\sum_i
\Big[
\theta_i^{(\mathrm{div})}\,\partial_y \phi_i(\mathbf{x})
+
\theta_i^{(\mathrm{curl})}\,\partial_x \phi_i(\mathbf{x})
\Big],
\quad
q_{y,NN}(\mathbf{x})
=
\sum_i
\Big[
-\theta_i^{(\mathrm{div})}\,\partial_x \phi_i(\mathbf{x})
+
\theta_i^{(\mathrm{curl})}\,\partial_y \phi_i(\mathbf{x})
\Big].
\]
In 3D, the curl-free counterpart is constructed in the same way as the 2D case.  
However, for the divergence-free counterpart, we use vector potentials for the construction. 
For each $\phi_i$, we define \(\mathbf{F}_i^{(x)}(\mathbf{x})=\mathbf{e}_x\,\phi_i(\mathbf{x})\), \(\mathbf{F}_i^{(y)}(\mathbf{x})=\mathbf{e}_y\,\phi_i(\mathbf{x})\), and \(\mathbf{F}_i^{(z)}(\mathbf{x})=\mathbf{e}_z\,\phi_i(\mathbf{x})\), with \(\mathbf{e}_x=(1,0,0)^\top\), \(\mathbf{e}_y=(0,1,0)^\top\), and \(\mathbf{e}_z=(0,0,1)^\top\).
We then take curls to obtain
\[
\boldsymbol{\phi}_{i}^{(\mathrm{div},x)}
=
\nabla\times \mathbf{F}_i^{(x)}
=
\begin{bmatrix}
0\\[2pt]
\partial_z \phi_i\\[2pt]
-\partial_y \phi_i
\end{bmatrix},
\,\,\,
\boldsymbol{\phi}_{i}^{(\mathrm{div},y)}
=
\nabla\times \mathbf{F}_i^{(y)}
=
\begin{bmatrix}
-\partial_z \phi_i\\[2pt]
0\\[2pt]
\partial_x \phi_i
\end{bmatrix},
\,\,\,
\boldsymbol{\phi}_{i}^{(\mathrm{div},z)}
=
\nabla\times \mathbf{F}_i^{(z)}
=
\begin{bmatrix}
\partial_y \phi_i\\[2pt]
-\partial_x \phi_i\\[2pt]
0
\end{bmatrix},
\]
each of which is divergence-free by construction. 
The divergence-free and curl-free parts of the flux are then represented as
\[
\mathbf{q}_{NN}^{(\mathrm{div})}(\mathbf{x})
=
\sum_i
\Big(
\theta_i^{(\mathrm{div},x)}\,\boldsymbol{\phi}_{i}^{(\mathrm{div},x)}(\mathbf{x})
+
\theta_i^{(\mathrm{div},y)}\,\boldsymbol{\phi}_{i}^{(\mathrm{div},y)}(\mathbf{x})
+
\theta_i^{(\mathrm{div},z)}\,\boldsymbol{\phi}_{i}^{(\mathrm{div},z)}(\mathbf{x})
\Big),
\]
\[
\mathbf{q}_{NN}^{(\mathrm{curl})}(\mathbf{x})
=
\sum_i \theta_i^{(\mathrm{curl})}\,\boldsymbol{\phi}_i^{(\mathrm{curl})}(\mathbf{x}).
\]

\textbf{Effect.}
We assess the effect of the physics-conforming construction using the setups provided in Fig.~\ref{fig:fig2}a and Fig.~\ref{fig:fig3}a.
As summarized in Table~\ref{tab:E1}, we report (i) the conservation residuals from the least-squares solve, (ii) relative $L_2$ errors or differences for pressure ($p$), velocity ($u_x$ and $u_y$), and concentration ($c$), and (iii) spectral discrepancy metrics, namely the relative spectral error and the Kolmogorov--Smirnov (K--S) distance. The conservation residual is evaluated from the fitted least-squares residual vector \(\mathcal{E}=\mathbf{A}\boldsymbol{\theta}^*-\mathbf{b}\) by extracting the entries associated with the mass-balance equations and computing the normalized quantity \(\mathcal{E}_{\mathrm{rel}}^{\mathrm{cons}}=\|\mathcal{E}_{\mathrm{cons}}\|_2^2/\|\mathbf{b}\|_2^2\). For the heterogeneous case, we report relative $L_2$ differences rather than errors, since no mesh-converged ground truth is available for the discrete-valued heterogeneous field and the finite-volume solution on the same spatial resolution is used as the reference baseline. Overall, Table~\ref{tab:E1} shows that enforcing the physics-conforming vector structure substantially reduces conservation residuals and improves solution quality, with especially strong gains in the heterogeneous setting, including clear reductions in concentration error and spectral mismatch. In incompressible settings, the benefit is even more immediate, since the divergence-free constraint is represented exactly by construction and can remove the need for a separate conservation equation in the least-squares system when source and sink terms are imposed as internal boundary conditions.

\begin{table}[ht!]
    \centering
    \caption{\bfseries Effect of physics-conforming neural vector bases.}
    \vspace{-0.8em}
    \label{tab:E1}
    \includegraphics[width=1\linewidth]{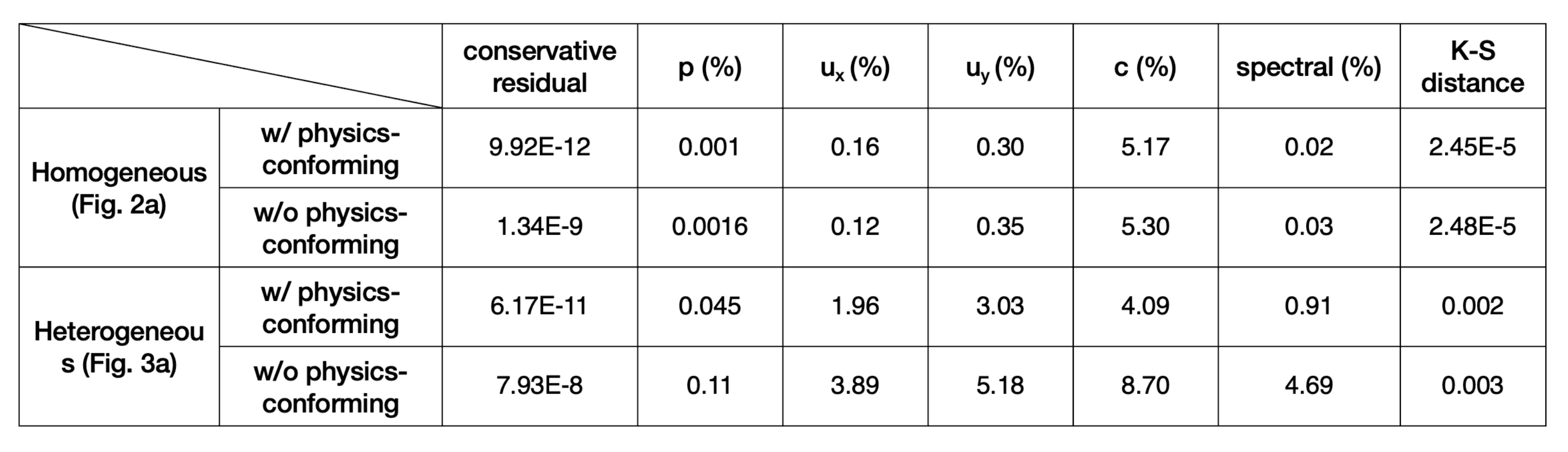}
\end{table}
\vspace{-15pt}

\section{Carbon storage example setups}\label{Appx:F}
\vspace{-1.5ex}
We consider a supercritical CO$_2$ storage benchmark in the square porous domain $\Omega=[0,762]~\mathrm{m}\times[0,762]~\mathrm{m}$ under two permeability settings: (i) a homogeneous case with constant permeability $\kappa=0.3~\mathrm{D}$ (Fig.~\ref{fig:fig2}a), and (ii) a heterogeneous case in which $\kappa(\mathbf{x})$ is given by a geologically realistic Gaussian random field (Fig.~\ref{fig:fig3}a), with permeability contrast $502$. A passive tracer is injected from the left boundary, and all other settings are identical between the two cases.

\textbf{Boundary conditions.}
On the left boundary ($x=0$), we prescribe pressure $p=103.4~\mathrm{bar}$ and a passive tracer pulse with concentration $c_{\mathrm{in}}=100~\mathrm{ppm}$ during the first 10 days. On the top boundary ($y=762~\mathrm{m}$), we impose an outward mass-flux $q_y=586~\mathrm{kg~day^{-1}~m^{-2}}$. The right boundary ($x=762~\mathrm{m}$) is no-flow ($q_x=0$), and the bottom boundary ($y=0$) is no-flow ($q_y=0$).

\textbf{Porous media and fluid properties.}
Besides the permeability field, other porous media and fluid properties are uniform: porosity $\varepsilon=0.25$, viscosity $\mu=0.04~\mathrm{cP}$, compressibility $c_f=2.90\times10^{-3}~\mathrm{bar}^{-1}$, and reference density $\rho_0=384~\mathrm{kg\,m^{-3}}$ at $50~^\circ\mathrm{C}$ and $p_0=103.4~\mathrm{bar}$.

\textbf{Reference and discretizations.}
For the homogeneous case, we use an finite-volume solution on a $1000\times1000$ structured uniform mesh as a ground-truth proxy. For the heterogeneous case, an analogous fine-grid ground truth is not available by refinement, since the permeability is specified through a discrete heterogeneous realization rather than a known continuous field. 
In both cases, the baseline finite-volume solution is computed on a $50\times50$ structured uniform mesh. 
For NBM, we use collocation on a matching $50\times50$ uniform interior grid (2500 interior points), together with 50 uniformly distributed collocation points on each boundary, for a total of 2700 collocation points. 
The vanilla PINN baselines use the same spatial resolution. 
For comparisons involving the energy-consistent weight \eqref{eq:ls-weight}, the baseline choice is uniform weighting with all weights set to 1.

\section{Manufactured-solution benchmark}\label{Appx:G}
\vspace{-1.5ex}
To quantify accuracy in a multiscale setting where grid refinement does not provide ground-truth proxy, we introduce a manufactured-solution benchmark based on the simplified incompressible system \(\mathbf{u}=-(\kappa(x,y)/\mu)\nabla p\) and \(\nabla\cdot\mathbf{u}=0\) on a regular domain \(\Omega\) with constant viscosity \(\mu\).
The permeability is chosen as a separable multiscale field $\kappa(x,y)=a(x)b(y)$, where the smooth oscillatory 1D factors are given by \(a(x)=\exp\!\left(s_x \frac{\varphi_x(x)}{\max|\varphi_x|}\right)\) and \(b(y)=\exp\!\left(s_y \frac{\varphi_y(y)}{\max|\varphi_y|}\right)\), with \(\varphi_x(x)=\beta_{x1}\sin(2\pi x)+\beta_{x2}\cos(4\pi x)\), \(\varphi_y(y)=\beta_{y1}\cos(2\pi y)+\beta_{y2}\sin(2\pi y)\), and \(s_x=\tfrac12\log(\mathcal{C}_x)\), \(s_y=\tfrac12\log(\mathcal{C}_y)\) controlling the 1D contrasts \(\mathcal{C}_x\) and \(\mathcal{C}_y\).
All permeability values are evaluated at cell centers of a uniform grid. 
An exact pressure field is then manufactured from the antiderivatives \(I_x(x)=\int_{x_{\min}}^{x}\frac{1}{a(s)}\,ds\) and \(I_y(y)=\int_{y_{\min}}^{y}\frac{1}{b(s)}\,ds\) by setting \(p(x,y)=p_0 + B\bigl(C_x I_x(x)+C_y I_y(y)\bigr)\), where $p_0$ is a reference pressure, $B$ sets the amplitude, and $C_x,C_y$ control the relative strengths of the two coordinate contributions. This gives \(\partial_x p=B\,\frac{C_x}{a(x)}\) and \(\partial_y p=B\,\frac{C_y}{b(y)}\), and hence the analytic velocity field \(u_x(x,y)=-\frac{B C_x}{\mu}\,b(y)\) and \(u_y(x,y)=-\frac{B C_y}{\mu}\,a(x)\), so that $u_x$ depends only on $y$, $u_y$ depends only on $x$, and the divergence-free constraint holds identically. Boundary conditions for $p$ are taken directly from the manufactured pressure, and the benchmark therefore admits an exact solution pair $(p,\mathbf{u})$ for any prescribed contrasts and coefficients.

In implementation, the permeability is assembled on cell centers as $\kappa_{ij}=k_0\,a(x_i)b(y_j)$ with uniform scaling factor $k_0$. The integrals $I_x$ and $I_y$ are evaluated on the corresponding 1D grids to construct $p_{ij}$ at cell centers, using either cumulative trapezoidal integration on the uniform grid or higher-accuracy piecewise Gauss--Legendre 8-point quadrature on each cell interval. The exact velocity is evaluated from the closed-form expressions above, avoiding numerical differencing.

For the experiments of Fig.~\ref{fig:fig3}k--s, we set $\Omega=[-1,1]\times[-1,1]$, $p_0=34~\mathrm{bar}$, $\mu=1.0~\mathrm{cp}$, $B=-20$, $(C_x,C_y)=(1.0,-0.6)$, and $k_0=0.1$. 
The permeability contrasts are sets as $\mathcal{C}_x=\max(a)/\min(a)=200$ and $\mathcal{C}_y=\max(b)/\min(b)=20$, yielding an overall contrast of about $1078$; varying $(\mathcal{C}_x,\mathcal{C}_y)$ generates fields with different contrast levels. This benchmark provides exact pressure and velocity only. The concentration reference is obtained by solving the transport equation with an finite-volume discretization using first-order upwinding driven by the manufactured velocity field, since an exact closed-form concentration solution is generally unavailable for heterogeneous advection under general initial and boundary conditions.

\section{NBM--OL implementation and experiment configuration}\label{Appx:H}
\vspace{-1.5ex} 
The NBM--OL implementation follows the structured workflow summarized in Fig.~\ref{fig:fig4}. We consider three operator-learning settings: \textbf{Case-I}, a permeability-varying steady Darcy flow operator; \textbf{Case-II}, a boundary-flux-varying Darcy flow operator; and \textbf{Case-III}, a boundary-flux/concentration-varying Darcy--transport operator. Across all cases, NBM--OL learns a map from the problem parameters (and time, when applicable) to solution coefficients in the neural basis space, enabling rapid inference of pressure, flux/velocity, and concentration fields. 

The common workflow is as follows: (i) define a parameterization of the operator family; (ii) generate coefficient snapshots by solving sampled training instances using NBM; (iii) apply POD blockwise to compress the coefficient snapshots into latent coordinates; (iv) train an MLP operator in the latent space using the residual metric $\mathcal{E}_{\mathrm{rel}}$ as a self-supervised objective; and (v) at inference, predict the latent coordinates, reconstruct the full coefficients, and recover the solution fields by neural basis expansion. 
When POD compression is not used, steps (ii) and (iii) are omitted.

All experiments use the square domain $\Omega=[0,762]~\mathrm{m}\times[0,762]~\mathrm{m}$. Interior residuals are evaluated on a structured $100\times100$ collocation set ($10{,}000$ points), and boundary residuals on 400 uniformly distributed boundary points with corners excluded. Pressure, mass-flux (or velocity in the incompressible case), and concentration are represented by separate neural bases, all with size $N_b=800$. When POD is used, it is applied separately to each coefficient block. OOD tests are constructed by extending one or more parameter axes and/or the time horizon beyond the training ranges.

\textbf{Case-I: permeability-varying steady Darcy flow operator.}
We impose mixed boundary conditions: $p=69~\mathrm{bar}$ on $x=0$, $u_x=0.3~\mathrm{m~day^{-1}}$ on $x=762~\mathrm{m}$, and $u_y=1.5~\mathrm{m~day^{-1}}$ on $y=0$ and $y=762~\mathrm{m}$, using the outward-normal sign convention. The base permeability template $\kappa_{\mathrm{ref}}(\mathbf{x})$ is defined on a $100\times100$ grid, and heterogeneity is parameterized by a piecewise-constant $5\times5$ blockwise field (block25) generated from a Gaussian random field. Each block25 realization is upsampled by block repetition and applied multiplicatively to $\kappa_{\mathrm{ref}}(\mathbf{x})$. We generate $N_s=1000$ realizations and compute steady NBM coefficient snapshots $(\boldsymbol{\theta}_{\mathrm{div}},\boldsymbol{\theta}_p)$; only the divergence-free velocity block is needed because this setup is incompressible. POD retains $r_p=16$ and $r_q=16$ modes, yielding $\mathbf{z}=[\mathbf{z}_p;\mathbf{z}_q]\in\mathbb{R}^{32}$. The latent map is learned by a two-branch MLP, one branch for pressure and one for velocity, each with two hidden layers of width 64 and $\tanh$ activation. Training uses Adam with learning rate $10^{-3}$, batch size 16, and 500 epochs. OOD tests are generated by perturbing the Gaussian random field controls, including variance/contrast scaling, correlation-length changes, and rotation-angle shifts.

\textbf{Case-II: boundary-flux-varying Darcy flow operator.}
The operator family is parameterized by a scalar top-boundary mass-flux control $g_{\mathrm{top}}$, uniformly sampled from $[488,976]~\mathrm{kg~day^{-1}~m^{-2}}$. We prescribe $p=103~\mathrm{bar}$ on $x=0$ and impose no-flux conditions on the right and bottom boundaries. The permeability field is fixed and multiscale on a $100\times100$ grid. We generate $N_s=200$ scenarios, each simulated for $N_{\mathrm{Darcy}}=20$ Darcy time steps with $\Delta T=10.0~\mathrm{days}$, yielding $N_{\mathrm{snap}}=4000$ snapshots for each coefficient block $(\boldsymbol{\theta}_{\mathrm{div}},\boldsymbol{\theta}_{\mathrm{curl}},\boldsymbol{\theta}_p)$. POD retains $r_{\mathrm{div}}=8$, $r_{\mathrm{curl}}=8$, and $r_p=8$ modes, giving $\mathbf{z}=[\mathbf{z}_{\mathrm{div}};\mathbf{z}_{\mathrm{curl}};\mathbf{z}_p]\in\mathbb{R}^{24}$. The latent operator $\mathcal{F}(g_{\mathrm{top}},t)$ is represented by three MLP branches, each with two hidden layers of width 32 and $\tanh$ activation. Training uses Adam with learning rate $10^{-3}$, batch size 16, and 500 epochs. OOD evaluation extends the $g_{\mathrm{top}}$ range by $50\%$ and the prediction horizon by $100\%$ beyond the training ranges.

\textbf{Case-III: boundary-flux/concentration-varying Darcy--transport operator.}
The coupled operator family is parameterized by $(g_{\mathrm{top}},c_{\mathrm{left}})$, where $g_{\mathrm{top}}$ is the top-boundary Darcy mass-flux control and $c_{\mathrm{left}}$ is the injected tracer concentration on the left boundary, sampled from $\mathcal{U}(50.0,100.0)~\mathrm{ppm}$. The Darcy boundary setup and fixed permeability field are the same as in Case-II unless stated otherwise. We generate $N_s=80$ independent scenarios. For each $g_{\mathrm{top}}$, the pretrained Case-II model is used to infer the Darcy velocity sequence, and each Darcy time step is split into $N_{\mathrm{tr}}=8$ transport substeps. We use $N_{\mathrm{Darcy}}=10$ Darcy coarse steps, so the total number of transport substeps is $N_{\mathrm{sub}}=N_{\mathrm{Darcy}}N_{\mathrm{tr}}$. POD is applied to the transport coefficient snapshots $\{\boldsymbol{\theta}_c\}$ and retains $r_c=20$ modes, yielding $\mathbf{z}_c\in\mathbb{R}^{20}$. The latent map $\mathbf{z}_c(t)=\mathcal{F}(g_{\mathrm{top}},c_{\mathrm{left}},t_{\mathrm{norm}})$ is learned by a single MLP with two hidden layers of width 64, $\tanh$ activation, and weight normalization. Training uses Adam with learning rate $5\times10^{-3}$, batch size 32, and 1000 epochs. OOD evaluation extends both $(g_{\mathrm{top}},c_{\mathrm{left}})$ by $50\%$ and the prediction horizon by $100\%$ beyond the training ranges.

\section{Physics-informed neural network (PINN) baseline}\label{Appx:I}
\vspace{-1.5ex} 
We use a two-network, two-stage vanilla PINN baseline for the coupled Darcy--transport problem. In Stage A, pressure is approximated by a neural network $p(\cdot;\omega_p)$ and trained by minimizing a weighted sum of mean-squared residuals at collocation points, consisting of the strong-form slightly compressible Darcy residual in the interior together with boundary penalties enforcing the left Dirichlet pressure condition and mass-flux conditions on the other boundaries, where the normal mass flux is obtained from $\nabla p$ via automatic differentiation. In Stage B, concentration is approximated by a second network $c(\cdot;\omega_c)$; after freezing $\omega_p$, we train $\omega_c$ using a conservative transport residual at space--time collocation points, together with penalties enforcing the initial condition $c(\cdot,0)=0$ and the time-dependent inlet Dirichlet condition on the left boundary. To reduce cost while preserving PDE-level coupling, once Stage A converges we precompute the mass-flux components and their divergence induced by $p(\cdot;\omega_p)$ on the full space--time collocation set and reuse them throughout Stage B, so that only $c(\cdot;\omega_c)$ is differentiated during transport training. 

Both networks are fully connected MLPs with $\tanh$ activations, width $64$, depth $4$, and scalar output, taking $(x,y,t)$ as input. Collocation uses a fixed $50\times 50$ spatial grid of cell centers and $N_t=10$ time levels, giving $50\times 50\times 10$ interior space--time points, while boundary collocation points are the boundary cell centers replicated across the same time levels. For Stage A, we use Adam with learning rate $10^{-3}$ for $20{,}000$ iterations with mini-batches of size $4096$ for interior points and $1024$ for each boundary set, followed by full-batch L-BFGS for $2{,}000$ iterations using strong Wolfe line search, history size $30$, $\texttt{tolerance\_grad}=10^{-10}$, and $\texttt{tolerance\_change}=10^{-12}$; the loss weights are $(w_{\mathrm{pde}}, w_L, w_R, w_B, w_T)=(1,1,1,1,10)$. For Stage B, we use Adam with learning rate $10^{-3}$ for $10{,}000$ iterations with mini-batches of size $4096$ for interior points, $1024$ for the initial condition, and $1024$ for the inlet condition, followed by full-batch L-BFGS for $2{,}000$ iterations with the same settings; the loss weights are $(w_{\mathrm{pde}}^{(c)}, w_{\mathrm{ic}}, w_{\mathrm{in}})=(1,0.01,10)$. For the multiscale cases, both Stage A and Stage B use $50{,}000$ Adam iterations followed by $20{,}000$ L-BFGS iterations.

\section*{CRediT author statement}
\vspace{-1.5ex} 
\textbf{Yuhe Wang}: Conceptualization, Methodology, Software, Validation, Formal analysis, Investigation, Visualization, Writing - Original Draft, Review \& Editing. \textbf{Min Wang}: Conceptualization, Methodoloty, Formal analysis, Investigation, Writing - Review \& Editing.

\section*{Data and Code Availability}
\vspace{-1.5ex} 
The code and data used in this study are available in the GitHub repository \url{https://github.com/neural-basis-method/nbm}

\bibliographystyle{unsrt}
\bibliography{references}

\end{document}